\documentclass{amsart}
\usepackage[margin=1in]{geometry}
\usepackage[colorlinks=true,citecolor=black,linkcolor=black,urlcolor=blue]{hyperref}
\usepackage{amsmath,pifont,enumitem,mathtools,shuffle}
\usepackage{tikz,pgf,hyperref}
\usepackage{physics}
\usepackage{mathdots}
\usepackage{yhmath}
\usepackage{siunitx}
\usepackage{array}
\usepackage{multirow}
\usepackage{diagbox}
\usepackage{standalone}
\usepackage[all]{xy}
\usepackage{colortbl}
\usepackage{caption}
\usepackage{thm-restate}
\usepackage{thmtools} 
\usepackage{float}

\usepackage{gensymb}
\usetikzlibrary{fadings}
\usetikzlibrary{patterns}
\usetikzlibrary{shadows.blur}
\usetikzlibrary{shapes}

\usepackage{amsmath}
\usepackage{amsthm}
\usepackage{amssymb}
\usepackage{mathtools}
\usepackage{soul}
\usepackage{color}

\newtheorem{theorem}{Theorem}[section]
\newtheorem{conjecture}[theorem]{Conjecture}
\newtheorem{proposition}[theorem]{Proposition}
\newtheorem{corollary}[theorem]{Corollary}
\newtheorem{lemma}[theorem]{Lemma}
\theoremstyle{definition}
\newtheorem{definition}[theorem]{Definition}
\newtheorem{example}[theorem]{Example}
\newtheorem{remark}[theorem]{Remark}
\newtheorem{question}[theorem]{Question}
\newtheorem{problem}[theorem]{Problem}

\newcommand{\out}[2]{\mathcal{O}_{#2}^{#1}}
\newcommand{\PF}{\mathrm{PF}}
\newcommand{\lucky}{\mathrm{lucky}}
\newcommand{\lace}{\mathrm{lace}}

\newcommand{\MPF}{\mathrm{MPF}}
\newcommand{\mpf}{\mathrm{mpf}}

\newcommand{\seqnum}[1]{\href{http://oeis.org/#1}{\underline{#1}}}

\usepackage[colorinlistoftodos]{todonotes}

\usepackage{amsmath}
\title{Metered Parking Functions}
\author[S.~Daugherty]{Spencer Daugherty}
\author[P.~E. Harris]{Pamela E. Harris}
\author[I.~Klein]{Ian Klein}
\author[M.~McClinton]{Matt ~McClinton}
\address[S. Daugherty, I. Klein]{Department of Mathematics, North Carolina State University, Raleigh, NC 27695}
\email{\textcolor{blue}{\href{mailto:spencer.page.daugherty@gmail.com}{sdaughe@ncsu.edu}}, \textcolor{blue}{\href{mailto:iklein@ncsu.edu}{iklein@ncsu.edu}}}
\address[P.~E. Harris, M. McClinton]{Department of Mathematical Sciences, University of Wisconsin-Milwaukee, Milwaukee, WI 53211}

\email{\textcolor{blue}{\href{mailto:peharris@uwm.edu}{peharris@uwm.edu}}, \textcolor{blue}{\href{mailto:mcclin33@uwm.edu}{mcclin33@uwm.edu}}}

\usepackage[backend=bibtex,maxbibnames=15, isbn=false, doi=false, url=false, giveninits=true]{biblatex}
\begin{document}

\begin{abstract} We introduce a generalization of parking functions called $t$-metered $(m,n)$-parking functions, in which one of $m$ cars parks among $n$ spots per hour then leaves after $t$ hours. We characterize and enumerate these sequences for $t=1$, $t=m-2$, and $t=n-1$, and provide data for other cases. 
We characterize the $1$-metered parking functions by decomposing them into sections based on which cars are unlucky, and enumerate them using a Lucas sequence recursion. Additionally, we establish a new combinatorial interpretation of the numerator of the continued fraction $n-1/(n-1/\cdots)$ ($n$ times) as the number of $1$-metered $(n,n)$-parking functions.
We introduce the $(m,n)$-parking function shuffle in order to count $(m-2)$-metered $(m,n)$-parking functions, which also yields an expression for the number of $(m,n)$-parking functions with any given first entry. As a special case, we find that the number of $(m-2)$-metered $(m, m-1)$-parking functions is equal to the sum of the first entries of classical parking function of length $m-1$. 
We enumerate the $(n-1)$-metered $(m,n)$-parking functions in terms of the number of classical parking functions of length $n$ with certain parking outcomes, which we show are periodic sequences with period $n$. 
We conclude with an array of open problems.

 \end{abstract}
\maketitle
\section{Introduction}

Throughout, let $n\in\mathbb{N}\coloneqq\{1,2,3,\ldots\}$,
and $[n]\coloneqq
\{1,2,\ldots,n\}$.
Given $m,n\in\mathbb{N}$ with $m\leq n$, consider the following parking scenario. 
There are 
$m$ cars in line to enter a one-way street consisting of $n$ parking spots. For each $i\in[m]$, where car $i$ prefers parking spot $a_i\in[m]$, a  complete list of parking preferences $\alpha=(a_1,a_2,\ldots,a_n)\in[n]^m$ is called a \textit{preference list}.
The cars enter the street in sequential order from 1 to $m$, and park according to the following \emph{parking scheme}: car $i$ drives to their preferred spot $a_i$, parking there if the parking spot $a_i$ is unoccupied. If the space is occupied, the car continues down the one-way street and parks in the first available spot it encounters, and if no such spot exists, then it exits the street without parking. 
Given the preference list $\alpha=(a_1,a_2,\ldots,a_m)\in[n]^m$, if (under this parking scheme) all cars are able to park in the first $n$ parking spots, then $\alpha$ is a called an $(m,n)$-parking function.  
For example, $\alpha=(7,5,3,3,2)$ is a $(5,7)$-parking function as, under the parking scheme, cars $1,2,3,4$ and $5$ park in spots $7,5,3,4$ and $2$, respectively, 
which we illustrate in Figure~\ref{fig:75example}. 

\begin{figure}[H]
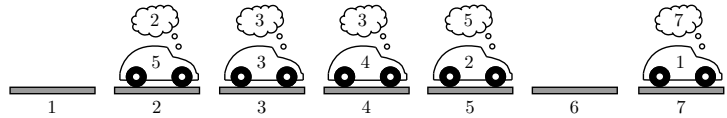

    \centering
    \tikzset{every picture/.style={line width=0.75pt}}
    \scalebox{.7}{
}
    \caption{The parking outcome under the preference list $(7,5,3,3,2)$. The number on the car is its original place in line and the number in the thought bubble above the car is that car's preferred parking spot.}
\label{fig:75example}
\end{figure}

We denote the set of $(m,n)$-parking functions by $\PF_{m,n}$ and, when $m=n$, the $(m,n)$-parking functions are referred to as \textit{parking functions} (of length $n$), and we denote the set by $\PF_n$. 
If $1\leq m\leq n$, then the cardinality of these sets satisfy the following formulas \cite{kenyon2023parking, Konheim1966AnOD}:
\[
    |\PF_{m,n}|=(n-m+1)(n+1)^{m-1} \qquad\mbox{and}\qquad |\PF_{n}|=(n+1)^{n-1}.\]
By the pigeon-hole principle, $\alpha=(a_1,a_2,\ldots,a_m)$ is an $(m,n)$-parking function if and only if 
$$ |\{k\in [m] : a_k \leq i \}| \geq m-n+i,$$ for all $i \in [n]$. Moreover, $\alpha \in \PF_{m,n}$ if and only if the non-decreasing rearrangement of $\alpha$, call it $\alpha'=(a_1',a_2',\ldots,a_m')$ has the property that $a_i' \leq n-m+i$ for each $i \in [m]$.

Konheim and Weiss introduced parking functions in \cite{Konheim1966AnOD}  and since then parking functions have been studied and connected to a variety of mathematical areas. 
For example, parking functions arise in the study of diagonal harmonics, in the computations of volumes of flow polytopes, in the enumeration of Boolean intervals in the weak order of Coxeter groups, in hyperplane arrangements, in ideal states in the game of the Tower of Hanoi, and in the sorting algorithm  $\mathrm{Quicksort}$ \cite{adenbaum2024boolean,Aguillon2022OnPF,Benedetti2018ACM, carlsson2020combinatorial, elder2023boolean, Garsia2012HLO, Haglund2003ACF,Harris2023LuckyCA, Hic2013PFP,Kim2015APF, LOEHR2005408, Mellit2021TBA, Stanley1996HyperplaneAI}.
In addition, many generalizations of parking functions have been developed and studied. 
This includes scenarios in which the parking scheme changes and allows cars to do something different whenever they find their preferred spot occupied \cite{Christensen2019AGO, countingKnaples,fang2024vacillating, MVP}, 
or cars' preferences are restricted to a subset or interval of parking spaces \cite{aguilarfraga2023interval,Hadaway2021GeneralizedPF,Spiro2019SubsetPF}, 
not all spots are available upon the cars arrival \cite{JoinSplit}, 
and cases where cars have varying sizes \cite{assortments,parkingsequences,Count_Assortments}. 
We note that $(m,n)$-parking functions as used here are a special case of the $\mathbf{x}$-parking functions introduced and studied by Pitman and Stanley, \cite{Stanley1996HyperplaneAI} and not the same as the $(m,n)$-parking functions of Aval and Bergeron \cite{aval2015interlaced}. 
For a survey of related results, we point the interested reader to \cite{yan2015parking} and for those interested in open problems in this area we recommend \cite{choose}. 

We introduce 
a new family of parking functions, parametrized by a positive integer $t$, that we call \textit{$t$-metered $(m,n)$-parking functions}. The new parking scheme imposes a maximum occupancy of the previous $t$ car(s) to park on the street. 
This can be equivalently interpreted as imposing a limit on the time that a car can remain parked in a parking spot, and hence the use of the name ``metered'' parking functions.
We make this precise next.
\begin{definition}\label{def: metered}
    Fix a positive integer $t$ and consider $m$ cars, each with a preferred spot, in line waiting to park in $n$ parking spots on the street. 
    Cars will enter one at a time and first attempt to park in their preferred spot before seeking forward for an unoccupied parking spot. 
    After car $j$ parks, car $j-t$ (if it exists) will immediately leave as it has exhausted its meter.
    If, given the preference list $\alpha\in[n]^m$, all cars are able to park on the street under these constraints, then $\alpha$ is a \emph{$t$-metered $(m,n)$-parking function}.
\end{definition}
For example, consider $t=1$, $m=3$, $n=2$, and 
$\alpha=(1,1,1)$. Under the $1$-metered parking scheme:
car~1 enters and parks in spot 1.
Car 2 enters and parks in spot 2, then car 1 immediately leaves. 
When car 3 enters the street, only car 2 is still parked, and hence, car 3 is able to park in spot 1. 
Thus, $\alpha$ is a $1$-metered parking function. We illustrate\footnote{In 1966, Konheim and Weiss~\cite{Konheim1966AnOD} proposed the ``capricious wives'' problem. As modern times have allowed us to reflect on the poor naming conventions of the past, we now refer to the problem as the ``parking function'' problem and appreciate the capriciousness of all motorists.} the process in Figure~\ref{fig:leaving}. On the other hand, $\beta = (1,1,2)$ is not a $1$-metered $(3,2)$-parking function because car $1$ parks in spot $1$,  car $2$ parks in spot $2$, car $1$ vacates spot $1$, then car $3$ is unable to park as spot $2$ is still occupied by car $2$ and there are no remaining spots available on the street.
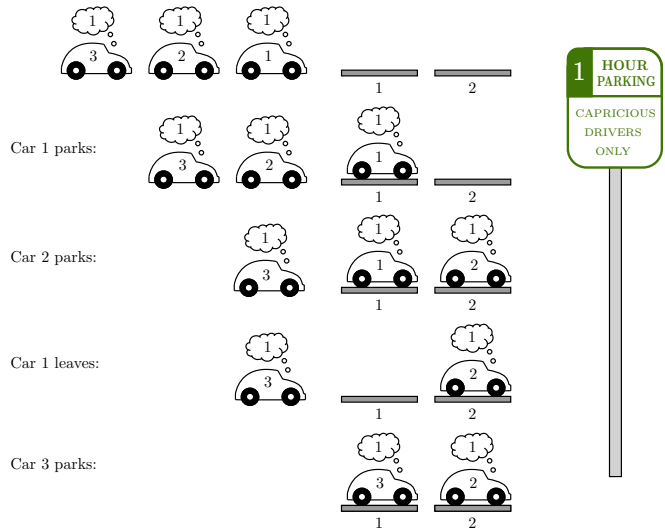
\begin{figure}[H]
    \tikzset{every picture/.style={line width=0.75pt}} 
\scalebox{.6}{
}
\end{figure}
As the example illustrates, in $t$-metered $(m,n)$-parking functions it is not necessary to restrict to the case $m\leq n$ as is needed in classical $(m,n)$-parking functions.
In fact, as (eventually) the number of cars on the street is always either $t$ or, during arrivals and departures, $t+1$, it suffices to have $t+1$ spots on the street so that any number of cars may be able to find parking on the street.

We let $\MPF_{m,n}(t)$ denote the set of $t$-metered $(m,n)$-parking functions, and let $\mpf_{m,n}(t) = |\MPF_{m,n}(t)|$ denote the cardinality of the set. 
For sake of simplicity, we say $t$-metered parking functions when $m$ and $n$ are clear from context or are arbitrary. 
As stated in the following result from Section~\ref{sec:prelim},
for certain values of $m,n,$ and $t$,
the metered parking scheme is equivalent to the classical parking scheme, and in this way the metered parking functions generalize classical parking functions.
\begin{restatable*}{proposition}{unmetered}\label{prop:unmetered}
    If $t \geq m-1$ and $n \geq m$, then $\MPF_{m,n}(t) = \PF_{m,n}$.
\end{restatable*}

In this work, we study $t$-metered parking functions and give complete characterizations and enumerations when $t=1$, $t=m-2$, and $t=n-1$. 
Before stating these results we remark that, unfortunately, the techniques utilized to characterize and enumerate $t$-metered parking functions for $t=1$, $t=m-2$, and $t=n-1$ do not generalize to other values of $t$. Hence, it remains an open problem to give enumerative formulas for $t$-metered $(m,n)$-parking functions for other values of $t$.

The first case we investigate is $t=1$. We prove that the number of $1$-metered $(m,n)$-parking functions satisfies the following recurrence, which also appears in Lucas numbers and Chebyshev polynomials.

\begin{restatable*}{theorem}{recursiontone}\label{thm:recursion t=1}
    For $m \leq n$, the number of $1$-metered parking functions satisfies the recursion
    $$\mpf_{m+1,n}(1) = n \cdot \mpf_{m,n}(1) - \mpf_{m-1, n}(1),$$ where $\mpf_{1,n}(1)=n$ and we use the convention that $\mpf_{0,n}(1) = 1$. 
\end{restatable*}
Setting $m=n$ in Theorem~\ref{thm:recursion t=1}, the sequence $(\mpf_1(n,n))_{n\geq 1}$, whose first 8 terms are highlighted as the main diagonal in Table~\ref{table_1mpf}, corresponds to the OEIS entry \cite[\seqnum{A097690}]{OEIS}. This sequence gives the numerators of the continued fraction \[n -  \frac{1}{n - \frac{1}{n - \frac{1}{n - \ldots}}},\]  which terminates after $n$ steps. Thus, our result gives a new combinatorial interpretation for these numerators, which satisfy the following closed formula given in OEIS entry \cite[\seqnum{A097690}]{OEIS}.

\begin{restatable*}{corollary}{closedtone}\label{cor:closed t=1}If $n >2 $, then
$$\mpf_{n,n}(1) = \sum_{k=0}^n (n-2)^{n-k} \binom{2n+1-k}{k}.$$    
\end{restatable*}

Our characterization of $(m-2)$-metered parking functions is done by giving a generalization of the classical parking function shuffle of Diaconis and Hicks \cite{DiaHic2017} and extending it to the $(m,n)$-parking function setting, see Theorem~\ref{thm:chartmtwo}. With that result at hand, we give the following enumerative formula.
\begin{restatable*}{theorem}{enumtmtwo}\label{thm:enumtmtwo}
    For $2 < m \leq n+1$, 
    \begin{align*}
    \mpf_{m,n}&(m-2)=(n-m+2)^{{2}}(m-1)^{m-3} \\
    &\qquad+\sum_{k=n-m+3}^n \binom{m-2}{n-k}(n-k+1)^{n-k-1}\Bigg[ \binom{k+1}{2}(n-m+2)k^{(m-n+k-3)} \\
    &\qquad\qquad+\left( k(n-m+1) - \binom{n-m+2}{2} \right)(k-n+m-1)^{k-n+m-3} \\
    &\qquad\qquad\qquad+ \sum_{j=n-m+2}^{k-1} \left(jk - \binom{j+1}{2}\right)\binom{m-2-n+k}{k-1-j}(n-m+1)j^{(j+m-2-n)}(k-j)^{k-j-2} \Bigg].
    \end{align*}
\end{restatable*} 

The values obtained by setting $n = m-1$ in Theorem~\ref{thm:enumtmtwo} have a special combinatorial interpretation.

\begin{restatable*}{corollary}{mtwocor}\label{cor:mtwocor}
    For $m\geq2$, the number of $(m-2)$-metered $(m,m-1)$-parking functions,  $\mpf_{m,m-1}(m-2)$, is equal to the sum of the first entry of all classical parking functions of length $m-1$. This count is given by $$\mpf_{m,m-1}(m-2) = \sum_{i=1}^{m-1} \sum_{s=0}^{m-1-i} \binom{m-2}{s}i(s+1)^{s-1}(m-s-1)^{m-s-3}.$$
\end{restatable*}

When considering $(n-1)$-metered $(m,n)$-parking functions, an interesting pattern emerges: the outcomes of these parking functions are periodic sequences with period $n$, see Theorem~\ref{thm:charn1}. 
We enumerate the $(n-1)$-metered $(m,n)$-parking functions in terms of the number of classical parking functions of length $n$ with certain outcomes, following Spiro \cite{Spiro2019SubsetPF}, and the values in those outcomes. In what follows, we let $\mathfrak{S}_n$ denote the set of permutations of $[n]$ and we write permutations in one-line notation.

\begin{restatable*}{corollary}{nplusk}\label{corr:nplusk}
Fix a permutation $\pi=\pi_1\pi_2\cdots\pi_n\in \mathfrak{S}_n$, and define $L_i(\pi)$ to be the largest $j$, with $1\leq j \leq i$, such that $\pi_i \geq \pi_N$ where $N \in \{i-j,i-j+1,\dots, i-1,i\}$. Then, for $k>0$, $$\mpf_{n+k,n}(n-1) = \sum_{\pi \in \mathfrak{S}_n} \left( \prod_{i=1}^n L_i(\pi) \prod_{j=1}^k \pi_{j \mod{n}}\right).$$
\end{restatable*}

This work is organized as follows.
In Section~\ref{sec:prelim}, we provide preliminary results, including a necessary property of metered parking functions (Proposition~\ref{thm:characterization metered}) and a connection between classical parking functions and metered parking functions (Proposition~\ref{prop:unmetered}). 
In Section~\ref{sec:t=1}, we specialize to the case $t=1$ and provide a complete characterization  for $1$-metered $(m,n)$-parking functions in Theorem~\ref{thm:t1lace_decomp}, as well as enumeration when $m \leq n$ (Theorem~\ref{thm:recursion t=1}).
In Section~\ref{sec:t=m-2}, we characterize $(m-2)$-metered parking functions, defined and study the $(m,n)$-parking function shuffle, and prove Theorem~\ref{thm:enumtmtwo} and Corollary~\ref{cor:mtwocor}.
In Section~\ref{sec:t=n-1}, we characterize $(n-1)$-metered parking functions and establish a formula for their enumeration using Theorem~\ref{thm:charn1} and Corollary~\ref{corr:nplusk}. 
Section~\ref{sec:future} outlines open questions and potential avenues for future work. We conclude with Appendix~\ref{ap:data}, where we provide additional data for the number of $t$-metered $(m,n)$-parking functions for $2\leq t\leq 4$ and $m,n\in [7]$ as well as the number of $t$-metered $(n,n)$-parking functions for $t,n \in [7]$.\footnote{
    New sequence data arising from our enumerations of metered parking functions is under review by the OEIS, with allocated numbers \cite[\seqnum{A372816}, \seqnum{A372817},\seqnum{A372818}, \seqnum{A372819},\seqnum{A372820}, \seqnum{A372821},\seqnum{A372822}]{OEIS}.}

\subsection*{Acknowledgements} The authors extend their thanks to the organizers of SaganFest 2024 for providing a positive environment ripe for building collaborations.
The authors thank Kimberly J. Harry for her suggestion of imposing a time limit to the length a car could remain in any parking spot, which motivated this work.
P.~E.~Harris was partially supported through a Karen Uhlenbeck EDGE Fellowship.

\section{Preliminary Results for $t$-Metered Parking Functions}\label{sec:prelim}

We investigate specific values of $t$ individually, but certain properties hold for many $t$ values. 
A key aspect of the definition in $t$-metered parking functions is that the first $t+1$ cars will first park on the street before the meter with length $t$ runs out,  at which point the first car must exit the street. If there are additional cars in queue, then the lowest numbered car that is currently parked on the street will exit, and the new car will attempt to park on the street, which will have exactly $t$ occupied spots. Hence, after the first $t+1$ cars parked, there will always be $t+1$ cars parked on the street prior to one car exiting and a new one entering.
Thus, we give a necessary condition for a preference list to be a $t$-metered $(m,n)$-parking function. 

\begin{proposition}\label{thm:characterization metered}
    If the preference list $\alpha=(a_1,a_2,\ldots,a_m)\in[n]^m$ is a $t$-metered $(m,n)$-parking function, then, for all $i\in [n]$ and any $j \in [m-t]$,  
\begin{align} 
    |\{ j \leq k \leq j+t : a_k \leq i \}| \geq  t+1-n+i.\label{char eq}
\end{align}
\end{proposition}

\begin{proof}

Consider within a preference list  $\alpha=(a_1,a_2,\ldots,a_m)$ any set of $t+1$ cars that may be parked on the street at the same time. 
By the pigeonhole principle, there can be at most one car that prefers the spot strictly greater than $n-1$, at most two cars that prefer spots strictly greater than $ n-2$, at most three cars that prefer spots strictly greater than $ n-3$, etc. In other words, there can be at most $n-i$ cars that prefer spots strictly greater than $i$ in each set of $t+1$ adjacent cars. Thus, there are at least $t+1-(n-i) = t+1-n+i$ cars that will prefer spots less than or equal to $i$ among any cars $a_j$ through $a_{j+t}$. 
\end{proof}

Unfortunately, the condition in Proposition~\ref{thm:characterization metered} is not a sufficient condition. Namely, it is not true that any preference list meeting the given inequality conditions is a $t$-metered parking function. We illustrate this next. 

\begin{example}
Consider $t=1, m=3$ and $n=2$ with $\alpha = (1,1,2)$.
We now show this preference list does satisfy the inequality conditions of Proposition~\ref{thm:characterization metered}. 
If $i=j=1$, then~\eqref{char eq} gives
$|\{ 1 \leq k \leq 2 : a_k \leq 1 \}|=|\{1,2\}|=2 \geq  {2}-2+1=1$.
If $i=1$ and $j=2$, then~\eqref{char eq} gives
$|\{ 2 \leq k \leq 3 : a_k \leq 1 \}|=|\{2\}|=1 \geq  {2}-2+1=1$.
Lastly, if $i=2$ and $j=2$, then~\eqref{char eq} gives
$|\{ 2 \leq k \leq 3 : a_k \leq 2 \}|=|\{2,3\}|=2 \geq  {2}-2+2=2$.
Thus, $\alpha$ meets the condition in Proposition~\ref{thm:characterization metered} but in the introduction we showed is not a $1$-metered parking function.
\end{example}

For many values of $t,m,$ and $n$, there are no $t$-metered $(m,n)$-parking functions.
\begin{proposition}
    If $m>n$ and $t \geq n$, then $\mpf_{m,n}(t) = 0$.
\end{proposition}
\begin{proof}
    Consider a preference list $\alpha\in[n]^m$ with $m>n$ and  $t\geq n$. For $\alpha$ to be a $t$-metered $(m,n)$-parking function, the first $n$ cars would need to be able to park in the $n$ spots on the street. Now, since $t\geq n$, each continues to occupy a spot as the $(n+1)$th car tries to park.
    Hence, no spots are available and thus $\alpha$ cannot be a $t$-metered $(m,n)$-parking function.
\end{proof}

On the other hand, many metered parking functions are also classical parking functions. If $t$ is large enough that no cars leave before the last car parks, then the parking scheme is effectively that of the classical parking functions. It is in this sense that metered parking functions are a generalization of classical parking functions.
We make this precise next.

\unmetered

\begin{proof}
    Let $\alpha=(a_1,a_2,\ldots,a_n)\in \PF_{m,n}(t)$. For $t\geq m-1$, as car $m$ with preference $a_m\in[n]$ attempts to park,  all previous cars continue to occupy  $m-1$ of the $n$ spots on the street. Thus, car $m$  parks if and only if $\alpha\in \PF_{m,n}$. The reverse set inclusion holds as all cars remain on the street when car $m$ enters to~park. 
\end{proof}

Note that proposition~\ref{prop:unmetered} implies that if $m-1 \leq t_1  < t_2$, then $\MPF_{m,n}(t_1) = \MPF_{m, n}(t_2)$. Additionally, an immediate consequence is that the first $t+1$ entries of any $t$-metered parking function will constitute a $(t+1,n)$-parking function.

\begin{remark}\label{t0_tm-1}
    The $0$-metered $(m,n)$-parking functions can be interpreted as a street consisting solely of a no-parking loading zone.
    Namely, the set of $0$-metered $(m,n)$-parking functions
    are simply sequences of length $m$ of the numbers $1$ through $n$. Thus, $\mpf_{m,n}(0) = n^m.$  
\end{remark}

One difficulty in working with $t$-metered parking functions is that they are not permutation invariant, meaning that rearranging the entries of a $t$-metered parking function does not necessarily yield another $t$-metered parking function. 
To make this precise, let $\sigma$ be a permutation on $[m]$, $\alpha=(a_1,a_2,\ldots,a_m)\in[n]^m$, and define
\[\sigma(\alpha)=(a_{\sigma^{-1}(1)},a_{\sigma^{-1}(2)},\ldots,a_{\sigma^{-1}(m)}).\]
For a classical parking function $\alpha\in\PF_{n,n}$, it is true that $\sigma(\alpha)\in\PF_{n,n}$ for all permutations $\sigma$ of $[n]$. However, this is not true for metered parking functions. In fact, even applying cyclic shifts to the entries of a metered parking function need not result in a metered parking function. We give an example of this behavior next.

\begin{example}
    Consider the $2$-metered $(4,6)$-parking function $\alpha = (6,2,5,6)$. If we apply the permutation $\sigma = 2431$ we get $\sigma(\alpha) = (6,6,5,2) \not\in \MPF_{4,6}(2)$. If we perform a cyclic shift once to the left on $\alpha$, we get $(2,5,6,6)$ and one can readily compute that $(2,5,6,6) \not\in \MPF_{4,6}(2)$. 
\end{example}

Another challenge in enumerating $t$-metered parking functions is that they contradict ones' intuition. One may suspect that a $t$-metered parking function should be a $t'$-metered parking function whenever $t'<t$, but this is not true. 
Moreover, it is also not the case that if a preference list is a $t$-metered parking function, then it is also a $t'$-metered parking function when $t<t'$. To give concrete examples, we first define the \textit{outcome} of a $t$-metered $(m,n)$-parking function. 
\begin{definition}\label{def:outcome}
    Given $\alpha\in[n]^m$, under the $t$-metered parking scheme the \textit{outcome of} $\alpha$ is defined as \[\out{t}{n}(\alpha)=(p_1,p_2,\ldots, p_m),\]
    where $p_i=j$ denotes that car $i$ parked in spot $j$ and if car $i$ fails to park, then we write $p_i=\mathsf{X}$. 
\end{definition}
In other parking function settings \cite{countingKnaples}, the outcome of a preference list is a permutation of the number of cars. As Definition~\ref{def:outcome} indicates, this is not the case for $t$-metered parking functions.

\begin{example}
    Consider the preference list $\alpha=(3,3,3,3,4,5,6)\in[6]^7$. Under the $2$-metered parking scheme,  $\alpha$ has outcome $\out{2}{6}(3,3,3,3,4,5,6)= (3,4,5,3,4,5,6)$. 
    However, under the $1$-metered parking scheme, $\alpha$ has outcome $\out{1}{6}(3,3,3,3,4,5,6)=(3,4,3,4,5,6,\mathsf{X})$, and car $7$ fails to park. Thus $\alpha\in\MPF_{7,6}(2)$ but $\alpha\notin\MPF_{7,6}(1)$. 
    On the other hand, consider $\beta=(5,5,5)\in[6]^3$. Under the $1$-meter parking scheme $\beta$ has outcome $\out{1}{6}(5,5,5)=(5,6,5)$. However, under the $2$-meter parking scheme $\beta$ has outcome $\out{2}{6}(5,5,5)=(5,6,\mathsf{X})$, and car $3$ fails to park.
    Thus $\beta\in\MPF_{3,6}(1)$ but $\beta\notin\MPF_{3,6}(2)$.
\end{example}
We can show, however, that classical $(m,n)$-parking functions will be metered parking functions under any meter $t$.
\begin{proposition}
    For $m \leq n$, if $\alpha$ is an $(m,n)$-parking function, then $\alpha$ is a $t$-metered parking function for any $t\geq 0$.
\end{proposition}
\begin{proof}
   Let $\alpha=(a_1,a_2,\ldots,a_m) \in \PF_{m,n}$ and let $\mathcal{O}_n(\alpha)=(\pi_1, \pi_2, \ldots, \pi_m)$ be its outcome under the classical parking scheme. Hence $\pi_i=j$ denotes that car $i$ parked in spot $j$, and note that $\{\pi_1,\pi_2,\ldots,\pi_m\}\subseteq[n]$  with $\pi_i\neq\pi_j$ for all $i\neq j$.
   For some $t\geq 0$, let $\out{t}{n}(\alpha) = (p_1, p_2,\ldots, p_m)$ be the $t$-metered outcome of the preference list $\alpha$. 
   We show by induction on $i$ that $p_i \leq \pi_i$ for all $i \in [m]$. Regardless of the meter $t\geq 0$, the first car will always get its preferred spot. So the outcome for the first car under the classical or metered parking scheme are the same, which implies that  $p_1 = \pi_1 = a_1$. 
   Now, assume that $p_k \leq \pi_{k}$ and consider $i = k+1$. 
   Assume for contradiction that under the metered parking scheme car $k+1$ is forced to park in a spot beyond $\pi_{k+1}$, where $\pi_{k+1}$ is the spot where car $k+1$ parks under the classical parking scheme. 
   In other words, assume that $p_{k+1}>\pi_{k+1}$. 
   This implies that the spots numbered $s$ with $a_{k+1}\leq s\leq \pi_{k+1}$ are all occupied by cars with indices less than or equal to $k$ at the time that car $k+1$ parks under the metered parking scheme. 
   If car $j$ with $j\leq k$ has the outcome $p_j = \pi_{k+1}$, then by our inductive statement we have $p_j \leq \pi_j$ which means that $\pi_{j} \geq \pi_{k+1}$. 
   This implies that in the outcome of $\alpha$ under the classical parking scheme, car $j$ either parks in spot $\pi_{k+1}$ or, spot $\pi_{k+1}$ is already occupied when car $j$ attempts to park.
   This is impossible, because under the classical parking scheme that spot must remain empty until car $k+1$ parks in it. 
   By contradiction, this shows that $p_{k+1} \leq \pi_{k+1}$.  
   Thus, we have established by induction that our claim is true for all $i \in [m]$.  This implies that if under $\alpha$ each car can park using the classical parking scheme, then they will be able to park under the $t$-metered parking scheme.
\end{proof}

Despite the lack of set containment in general for metered parking functions as $t$ varies, our data suggests the following relationship between the number of $t$-metered $(m,n)$-parking functions for different values of $t$. 

\begin{conjecture} \label{conj:future work}
    If $t_1<t_2$, then
    $\mpf_{m,n}(t_2)\leq \mpf_{m,n}(t_1).$
\end{conjecture}

The concepts of luck and displacement play a role in our study of metered parking functions and hence we recall them now. 
Given a parking function $\alpha\in[n]^m$, car $i$ with preference $a_i$ is \emph{lucky} if it parks in spot $a_i$. 
If instead car $i$ with preference $a_i$ parks in spot $p_i>a_i$, then we say car $i$ is \textit{unlucky}.
We define the \textit{displacement} of car $i$ by 
$d_i(\alpha)=p_i-a_i$, which measures the distance from the car's preferred spot to where it actually parked. 
The displacement of $\alpha$ is defined by  
\[d(\alpha)=\sum_{i=1}^m d_i(\alpha)=\sum_{i=1}^{m}(p_i-a_i),\]
where, for all $i\in[m]$, car $i$ has preference $a_i$ and parks in spot $p_i$.
The lucky statistic of classical parking functions has been studied in \cite{gessel2006refinement} and displacement in \cite{knuth1998linear, yan2015parking}. 
We extend these definition to $t$-metered parking functions analogously: For $\alpha \in \MPF_{m,n}(t)$, let $\lucky(\alpha)$ be the number of cars in $\alpha$ that park in their preference, and $d(\alpha)$ be the sum of the displacement each car's final parking position from its preference. 
Due to the time limit in our metered settings, there are certain bounds on displacement and luck that did not appear in the classical parking function case.
\begin{proposition}\label{prop:disp} 
Let $m,n\in \mathbb{N}$ and $\alpha = (a_1, a_2,\ldots, a_m) \in [n]^m$. Under the $t$-metered parking scheme, if car $i$ can park then $d_i(\alpha) \leq t$ for all $i\in[m]$. 
Moreover, the only cars that may not be able to park are those whose preference satisfies $a_i \geq n-t+1$.
\end{proposition}

\begin{proof}
    Let $\alpha=(a_1,a_2,\ldots,a_m) \in [n]^m$ be a preference list. Under the $t$-metered parking scheme, 
    when car $i$ attempts to park in spot $a_i$, there will be exactly $t$ cars parked on the street.
    If those cars are parked in spots $a_i$ through $a_1 + t - 1$, then car $i$ will be forced to park in spot $a_i+t$. So $d_i(\alpha)=(a_i+t)-a_i=t$. 
    If the cars are not parked in those spots, then car $i$ will be able to park in one of the spots numbered $a_i$ through $a_1 + t - 1$ and so $d_i(\alpha)<t$. 
    Together both cases imply that $d_i(\alpha)\leq t$, as claimed.
    
    The only case where car $i$ might not be able to park is if some of the aforementioned spots do not exist. That happens precisely when $a_i \geq n-t+1$.
\end{proof}

Note that one implication of Proposition~\ref{prop:disp} is that any preference list $\alpha \in [n-t]^m$ will be a $t$-metered $(m,n)$-parking function.
\begin{corollary}
    For any $n,m,t\in\mathbb{N}$,  $\mpf_{m,n}(t) \geq (n-t)^m$.    
\end{corollary}

We conclude with a counting formula for the number of $t$-metered $(2,n)$-parking functions.

\begin{corollary}\label{cor:m2} 
    For any $t \geq 1$, $\mpf_{2,n}(t) = n^2-1.$
\end{corollary}
\begin{proof}
    Note that for two cars with preference in $[n]$, there are a total of $n^2$ possible preference lists. As no cars would ever leave prior to the next car attempting to park, for any $t\geq 1$, of the $n^2$ preference lists, the only preference list that would not allow both cars to park is $(n,n)$. Thus, $\mpf_{2,n}(t)=n^2-1$, as claimed. 
\end{proof}

\section{$1$-Metered Parking Functions}\label{sec:t=1}

In a $1$-metered parking function, each car parks while only the car immediately before it remains on the street. 
As a result, any unlucky car with preference $a_i\in[n]$, would be able to park in spot $a_i+1$, provided $a_i+1\leq n$. 
Thus the only way a car may not be able to park is if it prefers spot $n$, which it finds occupied by the only other car on the street.
The simplicity of this scheme yields many properties for the $1$-metered parking functions. We begin by considering the example below. 

\begin{example}
    The set of $1$-metered $(3,3)$-parking functions, with size $\mpf_{3,3}(1)=21$, is
    $$\MPF_{3,3}(1)=\left\{\begin{matrix}
    (1,1,1) ,(1,1,2) ,(1,1,3) ,(1,2,1) ,(1,2,2) ,(1,2,3) ,(1,3,1) ,(1,3,2) ,(2,1,1) ,(2,1,2) , 
    (2,1,3),\\
    (2,2,1) ,(2,2,2) ,(2,3,1) ,(2,3,2) ,(3,1,1) ,(3,1,2) ,(3,1,3) ,(3,2,1) ,(3,2,2) ,(3,2,3) 
    \end{matrix}\right\}.$$
\end{example}

We characterize $1$-metered $(m,n)$-parking functions by breaking them down into smaller pieces, where each of those pieces is associated with a lucky car. We define this next.

\begin{definition}
    Let $p\in\mathbb{N}$.
     A sequence of the form $$(p,p,p+1,p+2,\dots, p+\ell-2),$$ is called 
    \emph{lace} of length $\ell \geq 2$, and the singleton $(p)$ is called a lace of length $\ell = 1$.
\end{definition}

\begin{definition}\label{def:lace}
    For $0 < m \leq n$, consider $\alpha = (a_1,a_2, \ldots, a_m)\in[n]^m$. The \textit{lace decomposition} of $\alpha$, denoted $\lace(\alpha)$, is defined iteratively by partitioning $\alpha$ (from left to right) into maximally long laces consisting of consecutively indexed entries in $\alpha$.
\end{definition}

\begin{example}
    If  $\alpha = (3,3,3,3,4,6)$, then we can partition $\alpha$ (from left to right) into three maximally long laces: $(3,3)$, $(3,3,4)$, and $(6)$. Hence the lace decomposition of $\alpha$ is given by  $\lace(\alpha)=((3,3),(3,3,4),(6))$. If $\beta=(4,1,2,2,3,4,5,5,1,2)$, then $\lace(\beta)=((4),(1),(2,2,3,4,5),(5),(1),(2))$.
\end{example}

\begin{remark}
    Definition~\ref{def:lace} is a generalization of the block decomposition of unit interval parking functions as studied in \cite{unit_pf}. 
    Namely, restricting to the set of tuples which are unit interval parking functions, which are classical parking functions where cars only tolerate parking in their preference or one unit away from their preference, Definition~\ref{def:lace} is precisely the definition of a block partition \cite[Definition 2.7]{unit_pf}. Moreover, Chaves Meyles et al.~\cite{meyles2023unitinterval} also studied unit interval parking functions in connection to the permutohedron and they also introduce a decomposition of a unit interval parking function into prime components and an operation called pipe. Our lace decomposition aligns with their decomposition into prime components, see \cite[Remark 2.7]{meyles2023unitinterval}.   
\end{remark}

\begin{lemma}\label{lem:characerization via laces}
    Let $0 < m \leq n$ and $\alpha = (a_1,a_2, \ldots, a_m)\in[n]^m$.
    Under the preference list $\alpha$ and $t = 1 $ parking scheme, car $i$ with preference $a_i$ is displaced by one if and only if $a_i$ is not the first entry of a lace in the lace decomposition of $\alpha$.
\end{lemma}
\begin{proof}
    First, note that under the $1$-metered parking scheme, for any $i\in[m]$, car $i$ is displaced by a maximum distance of 1 spot from its preference, and this displacement occurs exactly when car $i-1$, if it exists, parks in spot $a_i$.

    Now recall that by definition of the lace decomposition and the definition of a lace, if the first entry in a lace corresponds to a lucky car, then every other entry in that lace corresponds to a car that is displaced from their preference by one parking spot. 
    Namely, if we assume $a_i$ is the first entry in a lace and $a_j$ is the last entry in the lace, then the lace has structure $(a_i, a_i, a_i + 1, \dots, a_j - 1, a_j)$.
    So assuming that car $i$ parks in its preferred spot, every subsequent car in the lace will have to park one spot away from its preferred~spot. 

    Next, we want to show that a car whose preference is the first entry in a lace corresponds to a lucky car. Namely, if $a_i$ is the first entry in a lace, then car $i$ is lucky. 
    We show this by induction on $i$. In our base case, $i=1$, and we know that $a_1$ is always a lucky car, as it is the first car to enter the street.
    Assume the result holds for all $k$ with $1\leq k<i$ and we now consider what happens at index $i$. There are two cases to consider: either $a_{i-1}$  
    ends a lace of length one or it ends a lace of length greater than one.
    If $a_{i-1}$ ends a lace of length 1, then by our inductive hypothesis, car $i-1$ is lucky.
    Moreover, since $a_{i-1}$ ended a lace the lace decomposition of $\alpha$ ensures that $a_i \neq a_{i-1}$.
    This implies that car $i-1$ parks in spot $a_{i-1}$ while car $i$ can freely park in spot $a_i$, and hence it is lucky.
    On the other hand, if $a_{i-1}$ ends a lace of length at least 2, then by the construction of the lace decomposition, $a_{i} \neq a_{i-1} + 1$. 
    Furthermore, from our inductive hypothesis, we know that the first car in the lace containing car $i-1$ is a lucky car, which therefore tells us that car $i-1$ parks in spot $a_{i-1} + 1$. 
    Thus car $i$ can freely park in spot $a_i$, and is a lucky car, as claimed.

    Putting these facts together, we can now prove the claim.\\
    ($\Rightarrow$) If car $i$ is displaced by one spot, then $a_i$ cannot be the first entry in its lace, since the first entry in every lace corresponds to a lucky car.\\
    ($\Leftarrow$) If $a_i$ is not the first entry in its lace, then, since the first entry in its lace must correspond to a lucky car, car $i$ is displaced by one spot from their preference.
\end{proof}

We now give a characterization for $1$-metered parking functions based on the lace decomposition.

\begin{theorem}\label{thm:t1lace_decomp}
    Let $0 < m \leq n$ and $\alpha = (a_1,a_2, \ldots, a_m)\in[n]^m$.
    Then $\alpha\in\MPF_{m,n}(1)$ if and only if any instance of $n$ in $\alpha$ is located at the beginning of a lace in $\lace(\alpha)$.
\end{theorem}

\begin{proof}
    This follows from Lemma~\ref{lem:characerization via laces}, since each such sequence is a $1$-metered $(m,n)$-parking function if and only if no car preferring $n$ is displaced.
\end{proof}

Utilizing Theorem~\ref{thm:t1lace_decomp} we can give a simple characterization of $1$-metered $(m,2)$-parking functions.
The enumeration of $1$-metered $(m,2)$-parking functions is presented in Proposition~\ref{prop:t1_n2_count}.

\begin{corollary}\label{1mpf_n2_descrip} 
    Let $m\geq 1$ and $\alpha=(a_1,a_2,\ldots,a_m)\in[2]^m$ with lace decomposition $\lace(\alpha)=(\ell_1,\ell_2,\ldots,\ell_k)$.
    Then $\alpha\in\MPF_{m,2}(1)$ if and only if any lace containing a 2 is a singleton.
\end{corollary}
 
 \subsection{Enumerative Results} 
Table~\ref{table_1mpf} provides data for the number of $1$-metered $(m,n)$-parking functions.  
The entries above the main diagonal (where $m\leq n$) follow a recurrence relation, see Theorem~\ref{thm:recursion t=1}, and we show that the main diagonal (when $n=m$) corresponds to the OEIS entry~\cite[\seqnum{A097690}]{OEIS} and the diagonal $n=m+1$ corresponds to the OEIS entry \cite[\seqnum{A097691}]{OEIS}, see Corollary~\ref{cor:closed t=1}.

\begin{table}[h]
    \centering
    \begin{tabular}{|c||c|c|c|c|c|c|c|}
    \hline
    \backslashbox{$m$ cars}{$n$ spots}
      & \cellcolor{white} \textbf{\quad 1 \quad} & \cellcolor{white} \textbf{\quad 2 \quad}  & \cellcolor{white}  \textbf{\quad 3 \quad}  & \cellcolor{white} \textbf{\quad 4 \quad} & \cellcolor{white} \textbf{\quad 5 \quad} & \cellcolor{white} \textbf{\quad 6 \quad} & \cellcolor{white} \textbf{\quad 7 \quad} \\\hline
     \hline
     \cellcolor{white} \textbf{1}& \cellcolor{lightgray!35} 1 & 2 & 3 & 4 & 5 & 6 & 7 \\
     \hline
     \cellcolor{white} \textbf{2} & 0 & \cellcolor{lightgray!35} 3 & 8 & 15 & 24 & 35 & 48 \\
     \hline
     \cellcolor{white} \textbf{3} & 0 & 4 & \cellcolor{lightgray!35} 21 & 56 & 115 & 204 & 329 \\
     \hline
     \cellcolor{white} \textbf{4} & 0 & 6 & 55 & \cellcolor{lightgray!35} 209 & 551 & 1189 & 2255 \\
     \hline
     \cellcolor{white} \textbf{5} & 0 & 8 & 145 & 780 & \cellcolor{lightgray!35} 2640 & 6930 & 15456\\
     \hline
     \cellcolor{white} \textbf{6} & 0 & 12 & 380 & 2912 & 12649 & \cellcolor{lightgray!35}40391 & 105937 \\
     \hline
     \cellcolor{white} \textbf{7} & 0 & 16 & 1000 & 10868 & 60606 & 235416 & \cellcolor{lightgray!35} 726103 \\
     \hline
    \end{tabular}
    \caption{Number of $1$-metered parking functions.}
    \label{table_1mpf}
\end{table}
\begin{lemma}\label{1m_last_spot_lemma}
    If $1\leq m \leq j < k \leq n$, then the number of 1-metered $(m,n)$-parking functions where the last car parks in spot $j$ is equal to the number of $1$-metered $(m,n)$-parking functions where the last car parks in spot $k$.
\end{lemma}
\begin{proof}
    Let $Z_i(m,n)$ be the number of $1$-metered $(m,n)$-parking functions where the last car (car $m$) parks in spot $i$, and let $z_i(m,n) = |Z_i(m,n)|$.  For a set of sequences $X$, let $\mathcal{A}^i(X)$ be the set containing each $x \in X$ with $i$ appended to the end, and note that $|X| = |\mathcal{A}^i(X)|$. We want to show (for a fixed $n$) that $z_j(m,n) = z_k(m,n)$ if $m \leq j < k \leq n$.  
    
    We proceed by induction on $m$. When $m=1$, we have $Z_j(1,n) = \{j\}$ and $Z_k(1,n) = \{k\}$ so $z_{j}(1,n) = z_{k}(1,n)$. Assume for induction that when $m=h<n$, our claim is true. That is, $z_{j}(h,n) = z_{k}(h,n)$ for any $h \leq j < k \leq n$. 
    
    We now consider $m= h+1$.
    Pick any $j,k$ such that $h+1 \leq j< k \leq n$. Every element of $Z_j(h+1,n)$ is a $1$-metered $(h+1,n)$-parking function where the last car (car $h+1$) parks in spot $j$.  
    Thus, each of these sequences either ends in $j$ or in $j-1$, and what precedes it  (namely the tuple consisting of the first $h$ entries) will inherently be a $1$-metered $(h,n)$-parking function. We denote these two disjoint subsets of $Z_j(h+1,n)$ as $S_j$ and $S_{j-1}$ respectively.
   
    We claim that $S_j = \mathcal{A}^j(\MPF_{h,n}(1)) \setminus \mathcal{A}^j(Z_{j}(h,n))$ because it contains all of the $1$-metered $(h,n)$-parking functions with a $j$ appended except for those in which the final car can not park in its preferred spot of $j$. 
    This only happens if the car immediately before the last car parked in spot $j$. 
    Hence, $S_j$ is given by taking all of the $1$-metered $(h,n)$-parking functions where the last car parks in $j$ and appending a $j$. 
    On the other hand, in $S_{j-1}$ the second to last car must park in spot $j-1$ so that the last car will park in spot $j$. Thus $S_{j- 1} = \mathcal{A}^{j-1}(Z_{j-1}(h,n))$. Since $S_j$ and $S_{j-1}$ are disjoint and comprise all of $Z_j(h+1,n)$, we have that $|Z_j(h+1,n)| = | \mathcal{A}^j(\MPF_{h,n}(1)) \setminus \mathcal{A}^j(Z_{j}(h,n))| + |\mathcal{A}^{j-1}(Z_{j-1}(h,n))|$. Therefore, 
    $$z_j(h+1,n) = \mpf_{h,n}(1) - z_j(h,n) + z_{j-1}(h,n).$$
    Using the same logic, $z_k(h+1,n) = \mpf_{h+1,n}(1) - z_{k}(h,n) + z_{k-1}(h,n)$. By our inductive assumption, $z_j(h,n) = z_{k}(h,n)$ and $z_{j-1}(h,n)=z_{k-1}(h,n)$ because $j, k \geq h+1$ implies that $j,k \geq h$ and $j-1, k-1 \geq h$. Therefore, $z_j(h+1,n) = z_{k}(h+1,n)$ and we have proved our claim.
\end{proof}
Next we give another supporting result which will play a role in our subsequent enumerative results. 
\begin{lemma}\label{lemma_mpf_last_count} 
    If $m \leq n$, then $\mpf_{m-1,n}(1)$ counts the number of  $1$-metered $(m,n)$-parking functions where the last car parks in spot $n$.
\end{lemma}

\begin{proof} 
    Consider $Z_i(m,n)$, $z_i(m,n)$, and $\mathcal{A}^i(X)$ as defined in the proof of Lemma~\ref{1m_last_spot_lemma}. We need to show that $\mpf_{m-1,n}(1) = z_n(m,n)$. Every parking function in $Z_n(m,n)$ either ends in an $n$ or an $n-1$ since these are the only cars that park in spot $n$.
    Let $R_n$ be defined so that $\alpha \in R_n$ if and only if $\alpha \in Z_n(m,n)$ and the last entry in $\alpha$ is $n$. Define $R_{n-1}$ analogously, but the last entry of each element is $n-1$. So, $z_n(m,n) = |R_n| + |R_{n-1}|$.

    A prefix (complete subsequence starting at the beginning) of any $t$-metered parking function is also a $t$-metered parking function, so each $\alpha \in R_n$ is a $1$-metered $(m-1,n)$-parking function with $n$ appended. 
    Further, any $1$-metered $(m-1,n)$-parking function whose last car does not park in spot $n$ will appear in $R_n$ with $n$ appended. That is, 
    \[R_n = \mathcal{A}^n\big(\MPF_{m-1,n}(1) \setminus Z_{n}(m-1,n)\big),\]
    and $|R_n| = \mpf_{m-1,n}(1) - z_{n}(m-1,n)$. Because the last car of each sequence in $R_{n-1}$ prefers spot $n-1$ but parks in spot $n$, the second to last car must park in $n-1$. 
    Again, we can consider the elements of $R_{n-1}$ in terms of prefixes to see that $R_{n-1} = \mathcal{A}^{n-1}({Z_{n-1}(m-1,n)})$, and so $|R_{n-1}| = z_{n-1}(m-1,n)$.  We have now shown that 
    \begin{align}
        z_n(m,n) = \mpf_{m-1,n}(1) - z_{n}(m-1,n) + z_{n-1}(m-1,n).\label{eq:almost there}
    \end{align} 
    Since $m-1 \leq n-1 < n $, Lemma~\ref{1m_last_spot_lemma} tells us that $z_{n}(m-1,n) = z_{n-1}(m-1,n)$. Substituting $z_{n}(m-1,n) = z_{n-1}(m-1,n)$ into~\eqref{eq:almost there}, yields  $z_n(m,n) = \mpf_{m-1,n}(1)$, which proves the claim.
\end{proof}

We are now ready to establish a recursive formula for the number of $1$-metered $(m,n)$-parking functions.
\recursiontone
\begin{proof}
    Consider the set $\MPF_{m+1,n}(1)$. Every element in this set is a $1$-metered $(m,n)$-parking function with an additional entry appended to the end. 
    Note that appending $i\leq n-1$ to any $\alpha \in \MPF_{m,n}(1)$ yields a $1$-metered parking function because the only car that might not be able to park is one that prefers spot $n$.  
    In fact, the only time that appending a new entry to the end of $\alpha\in\MPF_{m,n}(1)$ would cause a car to fail to park is if the final car in $\alpha$ parks in spot $n$ and the new entry appended to $\alpha$ is $n$. Thus 
    \[\MPF_{m+1,n}(1) = \cup_{i=1}^n \mathcal{A}^i(\MPF_{m,n}(1) )\setminus \mathcal{A}^n(Z_n(m,n)).\]
    
    Because each $\mathcal{A}^i(\MPF_{m,n}(1))$ is disjoint and $|\mathcal{A}^i(\MPF_{m,n}(1))| = \mpf_{m,n}(1)$, this implies $\mpf_{m+1,n}(1) = n \cdot \mpf_{m,n}(1) - z_n(m,n)$. By Lemma~\ref{lemma_mpf_last_count}, this is equivalent to $\mpf_{m+1,n}(1) = n \cdot \mpf_{m,n}(1) - \mpf_{m-1,n}(1)$.
\end{proof}

The recursion in Theorem~\ref{thm:recursion t=1} is a special case of the recursion used in \emph{Lucas Sequences}. The Lucas sequence of the first kind $U_n(P,Q)$ is defined by the recurrence relations $U_0(P,Q)=0$, $U_1(P,Q)=1$, and $$U_n(P,Q) = P \cdot U_{n-1}(P,Q) - Q \cdot U_{n-2}(P,Q),$$ for $n>1$ \cite{Rib1, Rib2}. By this recursion, the sequence $(\mpf_{m,n}(1))_{m \geq 0}$ for a fixed $n$ is equivalent to the Lucas sequence $U_{m+1}(n,1)$ for $m<n$. Many properties are inherited from the Lucas sequence. For example, $$\mpf_{m,n}(1) = \frac{a^{m+1}-b^{m+1} }{a-b} \text{\qquad where $a$ and $b$ solve \quad} x^2 - nx + 1 =0 \text{\quad when} \quad n>2.$$

The sequence $(\mpf_{n,n}(1))_{n \geq 1}$ corresponds to the $n^{th}$ terms of the Lucas sequences $L(n,1)$, which is counted in OEIS sequence \cite[\seqnum{A097690}]{OEIS}. This sequence is also the \emph{Chebyshev polynomial of the second kind} \cite{Charafi1992ChebyshevPA}, denoted $U(n,x)$, evaluated at $x = \frac{n}{2}$. The Chebyshev polynomial of the second kind $U(n,x)$ is defined by the recursion $$U(n,x) = 2xU(n,x) - U(n-1,x) \text{\quad with \quad} U(0,x) = 1, \quad U(1,x) = 2x.$$
This sequence has generating function $$\frac{1}{1-nx + x^2},$$ from which we now give a closed formula in the special case of $m=n$ following \cite[\seqnum{A097690}]{OEIS}.

\closedtone

Utilizing the recursive formula in Theorem~\ref{thm:recursion t=1} and solving the associated characteristic polynomial, yields the following.

\begin{corollary}
    If $m \leq n + 1$ and $n > 2$, then $$\mpf_{m,n}(1) = \frac{n(n+\sqrt{n^2-4})-2}{n(n+\sqrt{n^2-4})-4} \cdot \left(\frac{n + \sqrt{n^2 -4}}{2}\right)^m + \frac{n(n-\sqrt{n^2-4})-2}{n(n-\sqrt{n^2-4})-4} \cdot \left(\frac{n - \sqrt{n^2 -4}}{2}\right)^m.$$
\end{corollary}

\begin{proof}
    This result arises from looking at the recursive formula $\mpf_{m,n}(1) = n\cdot \mpf_{m-1,n}(1) - \mpf_{m-2,n}(1)$. Fixing $n$ and thinking of this as a recursive formula on $f(m)$ where $f(m) = \mpf_{m,n}(1)$, we observe that the characteristic polynomial of $f(m)$ is $x^2-nx + 1$, which has roots
    
    $$\frac{n \pm \sqrt{n^2 - 4}}{2}.$$
    
    Therefore, we can write a closed formula for $f(m)$ as 
    
    $$f(m) = \alpha\left(\frac{n +\sqrt{n^2 - 4}}{2}\right)^m + \beta \left(\frac{n - \sqrt{n^2 - 4}}{2}\right)^m,$$
    
    for some constants $\alpha$ and $\beta$. Furthermore, we know that $f(1) = n$, and, by Corollary~\ref{cor:m2}, we know $f(2) = n^2 - 1$. With this information, we can set up the following system of linear equations
    \begin{align*}
        \alpha\left(\frac{n +\sqrt{n^2 - 4}}{2}\right) + \beta \left(\frac{n - \sqrt{n^2 - 4}}{2}\right) &= n\\
        \alpha\left(\frac{n +\sqrt{n^2 - 4}}{2}\right)^2 + \beta \left(\frac{n - \sqrt{n^2 - 4}}{2}\right)^2 &= n^2-1.
    \end{align*}
    Evaluating this system yields 
    \[
    \alpha  = \frac{n(n+\sqrt{n^2-4})-2}{n(n+\sqrt{n^2-4})-4}\mbox{\quad and \quad}
    \beta = \frac{n(n-\sqrt{n^2-4})-2}{n(n-\sqrt{n^2-4})-4}, 
    \]
    from which the result follows.
\end{proof}
\begin{remark}
    The sequence $(\mpf_{n,n}(1))_{n \geq 1}$ also corresponds to the sequence of numerators of the continued fraction \[n -  \frac{1}{n - \frac{1}{n - \frac{1}{n - \ldots}}},\] which terminates after $n$ steps. Additionally, the sequence $(\mpf_{n-1,n}(1))_{n \geq 2}$ corresponds to the denominators of the same continued fraction.  A similar phenomenon was observed in \cite{fang2024vacillating}, where the number of $1$-vacillating parking functions of length $n$ was shown to be the numerator of the $n$th convergent of the continued fraction expansion of $\sqrt{2}$. 
\end{remark}

Much like the main diagonal, the diagonals directly above it are also given by evaluations of Chebyshev polynomials of the second kind, which follows from our overall recursion. The $n=m+2$ diagonal of Table~\ref{table_1mpf} corresponds to OEIS sequence \cite[\seqnum{A342167}]{OEIS} and the $n=m+3$ diagonal correspond \cite[\seqnum{A342168}]{OEIS}, which are evaluations of $U(n,\frac{n+2}{2})$ and $U(n, \frac{n+3}{2})$ respectively. The diagonal $n=m+1$ corresponds to $U(n,\frac{n+1}{2})$, which happens to be the same as the denominators described in the remark above.

It remains an open question to enumerate $1$-metered $(m,n)$-parking functions when $m>n$, we state this formally in Problem~\ref{open:m>n}. We can, however, enumerate $1$-metered $(m,2)$ parking functions for all $m$ with a separate closed formula which matches with OEIS sequence \cite[\seqnum{A029744}]{OEIS}.
\begin{proposition} \label{prop:t1_n2_count}
    For $m>2$, \[
    \mpf_{m,2}(1) =
    \begin{cases}
        2^{\frac{m+1}{2}} & \text{if $m$ is odd,}\\
        3 \cdot 2^{\frac{m}{2} - 1} & \text{if $m$ is even.}    
    \end{cases}\]
\end{proposition} 

\begin{proof}
    We claim that $\mpf_{m,2}(1) = 2 \cdot \mpf_{m-2,2}(1)$. This recurrence solves the closed formula above given that $\mpf_{1,2}(1) = 2$ and $\mpf_{2,2}(1) = 3$. 
    Recall the description of these parking functions from Corollary~\ref{1mpf_n2_descrip}. 
    The first $m-2$ entries of any $1$-metered $(m,2)$-parking function constitute a $1$-metered $(m-2,2)$-parking function. 
    We will count all of the possible ways of appending two additional entries, $a_{m-1}$ and $a_{m}$, to these sequences such that we obtain a $1$-metered $(m,2)$-parking function. 
    We can never append $a_{m-1}=a_{m}={2}$ as a $1$-metered $(m,2)$-parking function can never have two  consecutive 2's. 
    Hence, we only need to consider appending the entries $a_{{m-1}}=a_{m}=1$, or $a_{{m-1}}=1$ with $a_{m}=2$, or $a_{{m-1}}=2$ with $a_{m}=1$.

    Let $S_{\gamma}$ denote the subset of $1$-metered $(m-2,2)$-parking functions where the last two cars are in a collection of laces of the form $\gamma$. For this proof, we will use $L_1$ to denote a lace that ends in a $1$ and other laces will be written explicitly.

    We can write $\MPF_{m-2,2}(1) = S_{(1,1)} \cup S_{L_1(1)} \cup S_{(2)(1)} \cup S_{(1)(2)}$, where the sets are disjoint. For a set of sequences $X$, let $\mathcal{A}^{i,j}(X)$ denote the set of sequences obtained by appending each $x \in X$ with an $i$ then a $j$.
    Adding only $1$'s to an $\alpha \in \MPF_{m-2,2}(1)$ will not affect the lace decomposition of any 2's, and so $\mathcal{A}^{1,1}(S_{(1,1)}) \subseteq \MPF_{m,2}(1)$, $\mathcal{A}^{1,1}(S_{L_1(1)}) \subseteq \MPF_{m,2}(1)$, $\mathcal{A}^{1,1}(S_{(1)(2)}) \subseteq\MPF_{m,2}(1)$, and $\mathcal{A}^{1,1}(S_{(2)(1)}) \subseteq \MPF_{m,2}(1)$.   

    If a sequence $\alpha \in S_{(1,1)}$ ends in a lace $(1,1)$ then appending a 2 next will never create a $1$-metered $(m,2)$-parking function, regardless of the next number. 
    If we append a $1$ and then a $2$, the lace decomposition among the last four cars will be $(1,1)(1)(2)$, so $\mathcal{A}^{1,2}(S_{(1,1)}) \subseteq \MPF_{m,2}(1)$. On the other hand, if a sequence $\alpha \in S_{L_1(1)}$ ends in laces $L_1(1)$, appending a $1$ and then a $2$ will never create a $1$-metered $(m,2)$-parking function because the final lace would be $(1,1,2)$.  If we append $2$ and then $1$, the lace decomposition of the final four entries will be $L_1(1)(2)(1)$, so $\mathcal{A}^{2,1}(S_{L_1(1)}) \subseteq \MPF_{m,2}(1)$.

    If a sequence $\alpha \in S_{(1)(2)}$ ends in laces $(1)(2)$ then we will never create a $1$-metered $(m,2)$-parking function by appending another $2$, regardless of what the following entry is. 
    We can, however, append a $1$ followed by a $2$ so that the lace decomposition of the new final four entries will be $(1)(2)(1)(2)$. Thus, $\mathcal{A}^{1,2}(S_{(1)(2)}) \subseteq \MPF_{m,2}(1)$. 
    If $\alpha \in S_{(2)(1)}$  ends in laces $(2)(1)$ then we cannot create a $1$-metered $(m,2)$ parking function by appending $1$ and then $2$ because the final lace decomposition would end with $(2)(1,1,2)$.  Instead, we can append a $2$ followed by a $1$ to get a lace decomposition that ends with $(2)(1)(2)(1)$. Thus, $\mathcal{A}^{2,1}(S_{(2)(1)}) \subseteq \MPF_{m,2}(1)$.

    We have considered all possibilities and so every entry in $\MPF_{m,2}(1)$ is in one of the aforementioned appended sets, all of which are disjoint. Given that $|\mathcal{A}^{i,j}(X)| = |X|$, we have 
    \[ \mpf_{m,2}(1) = 2|S_{(1,1)}| + 2|S_{L_1(1)}| + 2|S_{(1)(2)}| + 2|S_{(2)(1)}| = 2 \cdot \mpf_{m-2,2}(1). \qedhere\]
\end{proof}

Recall that by Corollary~\ref{cor:m2}, the number of $1$-metered $(2,n)$-parking functions is given by $n^2 - 1$. We conclude this section by providing a similar simple enumeration of $1$-metered $(3,n)$-parking functions, which aligns with OEIS entry \cite[\seqnum{A242135}]{OEIS}. 
\begin{proposition}\label{Prop:t1row3} 
    For $n\geq 2$, we have $\mpf_{3,n}(1) = n^3 - 2n.$
\end{proposition}

\begin{proof} 
    There are $n^3$ possible preference lists of length $3$ with entries in $[n]$; we will count the number that are not in $\MPF_{3,n}(1)$.  
    If $1\leq j < n$, then there are $n-1$ sequences of the form $(n,n,j)$ and there are $n-1$ of the form $(j,n,n)$, none of which can park, as either the second or third card would not find a spot on the street. 
    Additionally, the sequences $(n-1,n-1,n)$ and $(n,n,n)$ cannot park. By Theorem~\ref{thm:t1lace_decomp}, all of the other sequences considered will be able to park. 
    Therefore, $\mpf_{3,n}(1)=n^3 - 2(n-1)+2=n^3-2n$.
\end{proof}

\section{$(m-2)$-Metered Parking Functions}\label{sec:t=m-2}

The $(m-2)$-metered parking functions are an interesting special case because exactly one car will leave during the parking process. 
In a sense, these preference lists are as close to classical parking functions as possible without actually being classical parking functions. 
\begin{example}   
    Consider the preference sequence $\alpha = (3, 4, 1,1, 4, 2)$ where $m=6$ and $n=7$. Under the classical parking scheme, $\alpha$ has outcome $\mathcal{O}_7(\alpha) = (3,4, 1,2,5,6)$. In the $(m-2)$-metered setting, all but the last car will have exactly the same outcome, but now the last car will be able to park in the spot left open by the first car, which will leave after the second to last car arrives. Hence, $\out{4}{7}(\alpha) = (3,4,1,2,5,3)$.
\end{example}

To characterize and enumerate $(m-2)$-metered parking functions, we generalize the definition of a parking function shuffle as given by Diaconis and Hicks \cite[Page 129]{DiaHic2017}\footnote{We note that the definition of a parking shuffle presented in \cite[Page 129]{DiaHic2017} incorrectly states $\beta+(k-1)^{n-k}$ where it should instead say $\beta+(k)^{n-k}$.}
to the setting of $(m,n)$-parking functions. 
To begin, we recall again that $\alpha\in[n]^m$ is an $(m,n)$-parking function if and only if for all $k\in [n]$, 
\begin{align}
    |\{i \in[m] : a_i \leq k \}| \geq k + m-n.\label{mn pfs inequality}
\end{align} 
Additionally, recall that $\alpha$ is said to be a \textit{shuffle} of two sequences (or words) $\beta$ and $\gamma$ if $\alpha$ is formed by interspersing the entries of $\beta$ and $\gamma$, where the original orders of $\beta$ and of $\gamma$ are unchanged. For example, the shuffles of \emph{my} and \emph{car} are \emph{mycar, mcyar, mcayr, mcary, cmyar, cmayr, cmary, camyr, camry, carmy}.

\begin{definition}
    A sequence $(\pi_2,\ldots,\pi_m)$
    is an $(m,n)$-\emph{parking function shuffle} if it is a shuffle
    of the two words $\alpha$ and $\beta + (k)^{n-k} = (b_1 + k, \ldots, b_{n-k} + k)$ where $\alpha \in \PF_{m-n+k-1, k-1}$ and $\beta \in \PF_{n-k, n-k}$. We let $Sh_m(k-1, n-k)$ denote the set of all shuffles of $\alpha \in \PF_{m-n+k-1, k-1}$ and $\beta \in \PF_{n-k, n-k}$.
\end{definition}

Note that the value of $k$ is uniquely defined as the highest numbered parking spot available after we park an $(m,n)$-parking shuffle. 
Also, note that $(m,n)$-parking function shuffles are sequences of length $m-1$. The reason for considering sequences of the form $(\pi_2, \ldots, \pi_m)$ will be clear from Definition~\ref{def:A set}, which we give after the following example.

\begin{example}\label{ex:shuff72172}
    Consider the sequence $(3,7,2,1,7,2)$ parking on $8$ spots, which has outcome $\mathcal{O}(3,7,2,1,7,2)=(3,7,2,1,8,4)$ where the highest unoccupied spot is spot $6$. We see that $(3,7,2,1,7,2)$ is a $(7,8)$-parking function shuffle where $k=6$ because it is a shuffle of $(3,2,1,2) \in \PF_{4,5}$ and $(7,7) = (1,1) + (6,6)$ where $(1,1) \in \PF_{2,2}$.
\end{example}
The next definition plays a key role in our subsequent results.
\begin{definition}\label{def:A set} 
    For a sequence $(\pi_2, \ldots, \pi_m) \in [n]^{m-1}$, let $A^n_{(\pi_2, \ldots, \pi_{m})} = \{j : (j, \pi_2, \ldots, \pi_m) \in \PF_{m,n} \}.$
\end{definition} 
Notice that if $k \in A^n_{(\pi_2, \ldots, \pi_m)}$ then $k_1 \in A^n_{(\pi_2, \ldots, \pi_m)}$ for all $1 \leq k_1 \leq k$.
Thus, $A^n_{(\pi_2, \ldots, \pi_m)} = [k]$ for some value $k$. 

\begin{example}\label{ex:set72172} 
    Let $n=8$ and consider the parking function shuffle $(\pi_2, \ldots, \pi_m) = (3,7,2,1,7,2)$ from Example~\ref{ex:shuff72172}.  Observe that $$(1,3,7,2,1,7,2)\in \PF_{7,8}, \quad (2,3,7,2,1,7,2)\in \PF_{7,8},  \quad (3,3,7,2,1,7,2)\in \PF_{7,8},$$
    $$(4,3,7,2,1,7,2) \in \PF_{7,8}, \quad (5,3,7,2,1,7,2)\in \PF_{7,8}, \quad (6,3,7,2,1,7,2)\in \PF_{7,8},$$
    while
    $$(7,3,7,2,1,7,2)\not\in \PF_{7,8},\mbox{ and } \quad (8,3,7,2,1,7,2) \not\in \PF_{7,8}.$$
    Therefore, $A^8_{(3,7,2,1,7,2)} = [6]$.
\end{example}

The following statement generalizes \cite[Theorem 1]{DiaHic2017}, which is the special case where $m=n$.

\begin{theorem}\label{thm:mnshuffle} 
    Let $m \leq n$.
    Then $A^n_{(\pi_2, \ldots, \pi_m)} = [k]$ if and only if $(\pi_2, \ldots, \pi_m) \in Sh_m(k-1,n-k)$. 
\end{theorem} 

\begin{proof}
    We proceed by establishing the following claims, where Claims 1 and 2 will establish the forward direction, while Claims 3 and 4 will establish  the converse. \smallskip
 
    \noindent \textbf{Claim 1:} If $(\pi_2, \ldots, \pi_m) \in Sh_m(k-1, n-k)$, then $\pi=(k,\pi_2, \ldots, \pi_m)$ is an $(m,n)$-parking function. 
 
    To prove Claim 1 we first consider cars with preferences $j<k$. We have $|\{2\leq i\leq m : \pi_i \leq j\}| \geq j + m - n$ because this portion of the sequence $(\pi_2,\ldots,\pi_m)$ all comes from some $\alpha\in\PF_{m-n+k-1,k-1}$ that meets those conditions. The added first car is the only car that prefers $k$ and so $|\{ i : \pi_i \leq k \}| = |\{ i : \pi_i \leq k-1\}| + 1 \geq k + m -n$. When $j>k$, we are dealing with some of the cars from sequence $\beta + (k)^{n-k}$, all of the cars from $\alpha$, and the additional added $k$ at the start of $\pi$. 
    So, \[|\{i : \pi_i \leq j\}| = |\{i : \pi \leq k\}| + |\{i: b_i \leq j-k, \text{ with } b_i \in \beta\}| \geq k + m - n + j - k = j +m - n.\] Thus, $(k, \pi_2, \ldots, \pi_m)$ is an $(m,n)$-parking function. \smallskip
    
    \noindent \textbf{Claim 2:} If $(\pi_2, \ldots, \pi_m) \in Sh_m(k-1, n-k)$, then $\pi' = (k+1, \pi_2, \ldots, \pi_m)$ is \textbf{not} an $(m,n)$-parking function. 
 
    To prove Claim 2, we note that every entry in $\pi'$ that comes from $\beta$ is greater than $k$, as is the first entry of $\pi'$ since it is equal to $k+1$. 
    The only remaining entries are the $m-n+k-1$ entries from $\alpha$, all of which are at most $k-1$. Thus, \[|\{i : \pi_i \leq k\}| = m-1-n+k \leq k + (m-n),\]
    which by Equation~\eqref{mn pfs inequality} confirms that $\pi'$ is not an $(m,n)$-parking function. \smallskip
 
    \noindent ($\Rightarrow$) Claims 1 and 2 show that if $(\pi_2, \ldots, \pi_m) \in Sh_m(k-1,n-k)$, then $A^{n}_{(\pi_2, \ldots, \pi_m)} = [k]$. \smallskip

    \noindent\textbf{Claim 3:} If $\pi=(k, \pi_2, \ldots, \pi_m)$ is an $(m,n)$-parking function, then there exists a subsequence of $\pi$ that is an $(k-1+m-n,k-1)$-parking function. 
    This holds since, for example, one could take the subsequence formed by the $k-1+m-n$ cars that park in the spots $1$ through $k-1+m-n$. \smallskip
    
    \noindent \textbf{Claim 4:} If $\pi = (k, \pi_2, \ldots, \pi_m)$ is an $(m,n)$-parking function and $\pi' = (k+1, \pi_2, \ldots, \pi_m)$ is not an $(m,n)$-parking function, then there exists a subsequence of $\pi$ of the form $\beta + (k)^{n-k}$ where $\beta \in \PF_{n-k,n-k}$. 

    To establish this claim, we first note that for $\pi$ to be an $(m,n)$-parking function while $\pi'$ is not an $(m, n)$-parking function, $\pi$ must satisfy Equation~\eqref{mn pfs inequality} while $\pi'$ does not. This is only the case if $\pi$ has exactly $k+m-n$ cars with preferences less than or equal to $k$ (including the first car) which means that $\pi'$ has exactly $k+m-n-1$ cars with preferences less than or equal to $k$.  If $\pi$ has $k+m-n$ cars with preferences less than or equal to $k$ it must have $m-(k+m-n) = n-k$ cars that prefer spots numbered greater than $k$.
    
    Let $\gamma$ be the subsequence of $n-k$ cars in $\pi$ with preference greater than $k$ and let $\beta = \gamma - (k)^{n-k} = (b_1, \ldots, b_{n-k})$. 
    In $\pi$, for $n\geq j>k$, we have $|\{i: \pi_i \leq j\}| = | \{i : b_i + k \leq j\}| + |\{ i : \pi_i \leq k \}|$ because $\pi$ must satisfy Equation~\eqref{mn pfs inequality}. Here, we are splitting all the cars with preferences less than or equal to $j$ into those with preferences less than or equal to $k$ and those with preferences greater than $k$.
    Then by rearranging and using bounding values for $|\{i:\pi_i \leq j\}|$ and $|\{ i:\pi_i \leq k\}|$ given by Equation~\eqref{mn pfs inequality}, we obtain $$|\{i : b_i + k \leq j\}| = |\{i: \pi_i \leq j\}| - |\{ i : \pi_i \leq k \}| \geq j - m + n - (k - n + m) = j-k.$$ 
    Then, we have $|\{ i : b_{i} + k \leq j\}| = |\{ i : b_{i} \leq j-k\}| \geq j-k.$ Therefore, $\beta$ is an $(n-k, n-k)$-parking function by Equation~\eqref{mn pfs inequality}.
    \smallskip

    \smallskip
    \noindent ($\Leftarrow$)
    Claims 3 and 4 show that if $A^n_{(\pi_2, \ldots, \pi_m)} = [k]$, then $(\pi_2, \ldots, \pi_m) \in Sh_m(k-1, n-k)$, and so we have proved the theorem.
\end{proof}

The number of $(m,n)$-parking function shuffles is calculated as follows.

\begin{lemma}\label{lem:shuff_size} For $m \leq n$,
     \begin{align*}
         |Sh_m(k-1, n-k)| &=
         \begin{cases}
             \binom{m-1}{n-k}(n-m+1)k^{(m-n+k-2)}(n-k+1)^{n-k-1}&\mbox{if }k > n-m +1\\
              m^{(m-2)}&\mbox{if }k = n-m+1\\
             0&\mbox{if }k<n-m+1.\\
         \end{cases}
     \end{align*}
\end{lemma}

\begin{proof}
     Let $k > n-m+1$.  
     Each element in $Sh_m(k-1, n-k)$ is a shuffle of a word $\alpha \in \PF_{m-n+k-1, k-1}$ and $\beta + (k)^{n-k}$ where $\beta \in \PF_{n-k,n-k}$. 
     We know that $$|\PF_{m-n+k-1, k-1}| = (n-m+k+1)k^{(m-n+k-2)} \text{\quad and \quad} \PF_{n-k, n-k} = (n-k+1)^{(n-k-1)}.$$  
     There are $\binom{m-1}{n-k}$ shuffles of two words of length $m-n+k-1$ and $n-k$. 
     Therefore, we obtain the desired count by taking the product $\binom{m-1}{n-k}|\PF_{m-n+k-1, k-1}|\cdot|\PF_{n-k, n-k}|$. 
     If $k = n-m+1$ then $Sh_m(k-1, n-k)$ is shuffles of $\alpha \in \PF_{0,n-m}$ and $\beta+(k)^{n-k}$ where $\beta \in \PF_{n-k,n-k}$. 
     So, we only need to count $\beta \in \PF_{n-k,n-k}$, noting that in this case $n-k = m-1$, which is $|\PF_{m-1,m-1}| = m^{(m-2)}$. 
     It is impossible for $k$ to be lower than $n-m+1$ because there would be more than $m-1$ cars required for $\beta$.
\end{proof}
We are now able to give a count for the number of $(m,n)$-parking functions with a specified first entry.

\begin{corollary}\label{cor:enum_start}
    For $m\leq n$, the number of $\alpha=(a_1,a_2,\ldots,a_m) \in \PF_{m,n}$ with $a_1 = j > n-m+1$ is $$\sum_{i=j}^n \binom{m-1}{n-i}(n-m-1)i^{i+m-n-2}(n-i+1)^{n-i-1},$$ and the number of $\alpha =(a_1,a_2,\ldots,a_m)\in \PF_{m,n}$ with $a_1 = j \leq n-m+1$ is $$ m^{(m-2)} + \sum_{i=n-m+2}^n \binom{m-1}{n-i}(n-m-1)i^{i+m-n-2}(n-i+1)^{n-i-1}.$$
 \end{corollary}

\begin{proof} 
    If $a_1 = j$ then $A^n_{(a_2, \ldots, a_m)} = [k]$ for some $k \geq j$. So, the number of parking functions that start with $j$ is the number of $A_{\pi} = [k]$ for any $k \geq j$. By Theorem~\ref{thm:mnshuffle}, this is equivalent to $\sum_{i=j}^n |Sh_m(i-1, n-i) |.$ This translates to our result by Lemma~\ref{lem:shuff_size}.
\end{proof}  

\begin{example}
   The $(3,4)$-parking functions that start with $3$ are:
   \begin{align*}
       &(3,1,1), \quad (3,1,2), \quad (3,1,3), \quad (3,1,4), \quad (3,2,1), \quad (3,2,2), \\
       &(3,2,3), \quad (3,2,4), \quad (3,3,1), \quad (3,3,2), \quad (3,4,1), \quad (3,4,2).
   \end{align*}
\end{example}

These are the sequences of the form $(3,\pi_2, \pi_3)$, where $A^4_{(\pi_2, \pi_3)} = [k]$ for $k \geq 3$. Next, we characterize the $(m-2)$-metered $(m,n)$-parking functions in terms of the $(m,n)$-parking shuffle.

\begin{theorem}\label{thm:chartmtwo} 
    For $2<m\leq n+1$, let $\alpha =(a_1,a_2,\ldots,a_{m})\in[n]^{m}$ 
    such that $(a_2, \ldots, a_{m-1}) \in Sh_{m-1}(k-1, n-k)$ for some $k \in [n]$, and let $(a_{i_1}, \ldots, a_{i_{m-n+k-1}})$ be the subsequence of entries in $(a_2, \ldots, a_{m-1})$ that are less than $k$. When $m \leq n$, also consider $j \in [k-1]$ such that $(a_{i_1}, \ldots, a_{i_{m-n+k-2}}) \in Sh_{m-n+k-1}(j-1, k-j-1)$, and when m=n+1 set $j=0$.
    Then, $\alpha \in \MPF_{m,n}(m-2)$ if and only if either both $a_1 \leq j$ and $a_m \leq k$, or if  both $j < a_1 \leq k$ and $a_m\leq a_1$.
\end{theorem}
Before we begin the formal proof of Theorem~\ref{thm:chartmtwo}, we outline the idea of the proof. Consider all of the given assumptions setting up the statement of Theorem~\ref{thm:chartmtwo}, with this set up, if we just park cars $2$ through $m-1$ in $\alpha$, then the highest numbered parking spot which is unoccupied will be spot $k$, and the second highest numbered parking spot which is unoccupied will be spot $j$. In the case that $ m=n+1$, then  we let $j=0$ because this spot does not exist.
When we park cars $1$ through $m$ instead, the following will happen:
If $a_1 \leq j$, then car 1 with preference $a_1$ will be able to park in a spot numbered at most $j$ and the cars numbered above $j$ will not be affected, meaning that car $m$ is guaranteed to park as long as it prefers a spot numbered less than or equal to $k$, as $k$ denotes the largest numbered empty spot on the street. If $j < a_1 \leq k$, then a car is guaranteed to park in spot $k$. 
So, the highest numbered open spot after car $1$ leaves the street will be the spot that car $1$ just vacated. 
Thus, in this case, we need $a_m \leq a_1$  in order for car $m$ to park.
We are now ready to prove Theorem~\ref{thm:chartmtwo}.

\begin{proof} 
    First we  show that every sequence $\alpha \in \MPF_{m,n}(m-2)$ has the form described in our claim. 
    Note  $\alpha = (a_1,a_2, \ldots, a_{m})$ is an $(m-2)-$metered $(m,n)$-parking function if and only if $(a_1,a_2, \ldots, a_{m-1})\in \PF_{m-1,n}$, which is true if and only if $a_1 \leq k$ for $A^n_{(a_2, \ldots, a_{m-1} )} = [k]$. 
    By Theorem~\ref{thm:mnshuffle}, this is true if and only if $a_1 \leq k$ and $(a_2, \ldots, a_{m-1}) \in Sh_{m-1}(k-1, n-k)$. If $m=n+1$ then, $(a_1,a_2, \ldots, {a_{m-1}})\in \PF_{n,n}$, so there will be no open spaces once the {first $m-1$} cars have parked. 
    If $m \leq n$, we consider the entries in $(a_2, \ldots, a_m)$ that are less than $k$, of which we know there are $m-n+k-2$ because they come from $Sh_{m-1}(k-1, n-k)$ which also tells us that $(a_{i_1}, \ldots, a_{i_{m-n+k-2}})\in \PF_{m-n+k-2, k-1}$.  
    Since $m \leq n$, we see $m-n+k-2<k-1$ and so there is some maximal $j$ that can be prepended so that $(j, a_{i_1}, \ldots, a_{i_{m-n+k-2}}) \in \PF_{m-n+k-1, k-1}$. 
    We apply Theorem~\ref{thm:mnshuffle} again to see that $(a_{i_1}, \ldots, a_{i_{m-n+k-2}})\in Sh_{m-n+k-1}(j-1, k-j-1)$. \\

    \noindent ($\Leftarrow$) Assume that we have a sequence $\alpha$ that meets the conditions described above.
    There are two cases to consider: (1) $a_1 \leq j$ and $a_m \leq k$, or (2) $j < a_1 \leq k$ and $a_m \leq a_1$. In each case we must establish that $\alpha \in 
    \MPF_{m,n}(m-2)$.

    Case (1): Note that this case can only occur when $m \leq n$. Let $a_1 \leq j$ and $a_m \leq k$. As stated, the cars with preferences $(a_{i_1}, \ldots, a_{i_{m-n+k-2}})$ are the only cars among those numbered $2$ through $m-1$ that prefer a spot numbered strictly less than $k$. 
    Notice that $(a_{i_1}, \ldots, a_{i_{m-n+k-2}}) \in Sh_{m-n+k-1}(j-1, k-j-1)$ implies that $(a_1, a_{i_1}, \ldots, a_{i_{m-n+k-2}})\in \PF_{m-n+k-1, k-1}$ because $a_1 \leq j$. So, none of these cars affect the parking positions of any of the cars that prefer spots above $k$. As a result, those cars will park as they would without car $1$ present and so spot $k$ will not be occupied. Therefore, car $m$ will park since at least one spot at, or above, its preference is available.     

    Case (2): Let $j<a_1\leq k$ and $a_m\leq a_1$, recalling that if $m = n+1$ we set $j=0$. 
    Since $(a_2, \ldots, a_{m-1})\in Sh_{m-1}(k-1, n-k)$, 
    we have $(a_1,a_2, \ldots, a_{m-1}) \in \PF_{m-1, n}$. Then, car $1$ will leave before car $m$ parks, so there will be an open spot for car $m$ regardless of the locations of the other cars. Since $a_m\leq a_1$, car $m$ will park in the spot vacated by car $1$, {or an open spot before it.} \\

    \noindent($\Rightarrow$) On the other hand, assume that $\alpha \in \MPF_{m,n}(m-2)$, meaning that  all the cars can park. 
    We know from the beginning of the proof that $a_1 \leq k$. Similarly, since no spot higher than spot $k$ is available when cars $1$ through $m-1$ park, we must have that $a_m \leq k$. We consider the following two cases: when $m=n+1$ and when $m\leq n$.

    When $m=n+1$, as shown above, every spot will be taken by a car after the first $m-1$ cars (those with preferences $(a_1, \ldots, a_{m-1})$) park. 
    Thus, car $m$ must have a preference satisfying $a_m \leq a_1$ so that car $m$ can take the spot that car $1$ vacates, which is spot $a_1$. 

    Now, for $m \leq n$, we only need show that if $j < a_1 \leq k$, then $a_m \leq a_1$. 
    Because $(a_{i_1}, \ldots, a_{i_{m-n+k-2}}) \in Sh_{m-n+k-1}(j-1,k-j-1)$, the subsequence of cars whose preferences range from $j+1$ to $k-1$ must park in the spots numbered $j+1$ through $k-1$, inclusive. 

    Note this is also true when we park the cars 2 through $m$ with preferences $(a_2,\ldots, a_m)$, as the subsequence of cars (whose preferences are in the range $j + 1$ to $k-1$) came from this original preference list.
    When car $1$ parks before these cars, it will take its preferred spot $a_1$ where $j<a_1\leq k$. 
    Thus, if $j < a_1< k$, then one of the cars that, in the absence of car $1$, would park in the spots numbered $j+1$ to $k-1$ will be forced to park in spot $k$ instead. Otherwise, if $a_1=k$, then the cars that previously parked in the spots numbered $j+1$ to $k-1$ will park in the exact same way as car 1 does not affect them. In both cases, spot $k$ will be occupied when cars $1$ through $m-1$ are parked.
    Recall, that by definition, $k$ is the largest empty spot, after cars $2$ through $m-1$ have parked. 
    As we showed now, whenever $j<a_1\leq k$, parking the cars $1$ through $m-1$, ensures that spots $k$ through $n$ are occupied. As we are in a $(m-2)$-metered parking function, when car $m$ enters the street, the only car that has left the parking lot is car 1, which parked in its preference $a_1$. Thus, in order to park, car $m$ must have a preference satisfying $a_m\leq a_1$.
\end{proof}
The following example illustrates the characterization given in Theorem~\ref{thm:chartmtwo}.
\begin{example} 
    Consider $(a_2, a_3,\ldots, a_{m-1}) = (3,6,4,7,1,7) \in Sh_{7}(5-1,8-5)$. Note that this is a shuffle of $(3,4,1)\in \PF_{3,4}$ and $(1,2,2)+(5)^{3}$, where $(1,2,2)\in\PF_{3,3}$.
    Also observe that $(3,4,1)\in Sh_{4}(2-1,4-2)$. 
    We now construct the possible $\alpha = (a_1, a_2, \ldots, a_{m-1}, a_m)$ such that $\alpha \in \MPF_{8,8}(m-2)$.
    By Theorem~\ref{thm:chartmtwo} with $k=5$ and $j=2$, the tuples created from $a_1\leq j$ and $a_m\leq k$ are:
    \begin{align*}
        &(1,3,6,4,7,1,7,1), \,(1,3,6,4,7,1,7,2), \,(1,3,6,4,7,1,7,3), \,(1,3,6,4,7,1,7,4),\,(1,3,6,4,7,1,7,5),\,\\
        &(2,3,6,4,7,1,7,1),\, (2,3,6,4,7,1,7,2),\, (2,3,6,4,7,1,7,3), \,(2,3,6,4,7,1,7,4),\,(2,3,6,4,7,1,7,5),\\
    \intertext{while the tuples created from $j<a_1\leq k$ and $a_m\leq a_1$ are:}
        &(3,3,6,4,7,1,7,1),\, (3,3,6,4,7,1,7,2), \,(3,3,6,4,7,1,7,3),\\
        &(4,3,6,4,7,1,7,1), \,(4,3,6,4,7,1,7,2), \,(4,3,6,4,7,1,7,3),\,(4,3,6,4,7,1,7,4),\\
        &(5,3,6,4,7,1,7,1), \,(5,3,6,4,7,1,7,2), \,(5,3,6,4,7,1,7,3),\,(5,3,6,4,7,1,7,4), \,(5,3,6,4,7,1,7,5).
    \end{align*}
\end{example}
We now provide an enumeration for the set of $(m-2)$-metered $(m,n)$-parking functions.

\begin{table}[!h]
    \centering 
    \begin{tabular}{|c||c|c|c|c|c|c|c|}
    \hline 
    \backslashbox{$m$ cars}{$n$ spots} & 
    \textbf{\quad 1 \quad} &\cellcolor{white}\textbf{\quad 2\quad } & \cellcolor{white}\textbf{\quad 3\quad} & \cellcolor{white}\textbf{\quad 4\quad} & \cellcolor{white}\textbf{\quad 5\quad} & \cellcolor{white}\textbf{\quad 6\quad } & \textbf{\quad 7 \quad}\\
    \hline 
    \hline
    \textbf{1} & \cellcolor{lightgray!35}0 & 0 & 0 & 0 & 0 & 0 & 0 \\ \hline
    \cellcolor{white}\textbf{2}& 1 & \cellcolor{lightgray!35}4 & 9 & 16 & 25 & 36 & 49 \\
    \hline 
    \cellcolor{white}\textbf{3}& 0 & 4 & \cellcolor{lightgray!35}21 & 56 & 115 & 204 & 329 \\
    \hline 
    \cellcolor{white}\textbf{4}& 0 & 0 & 27 & \cellcolor{lightgray!35}163 & 483 & 1095 &2131 \\
    \hline 
    \cellcolor{white}\textbf{5}& 0 & 0 & 0 & 257 & \cellcolor{lightgray!35}1686 & 5367 & 13076\\
    \hline 
    \cellcolor{white}\textbf{6}& 0 & 0 & 0 & 0 & 3156 & \cellcolor{lightgray!35}21858 & 73276 \\
    \hline 
    \cellcolor{white}\textbf{7}& 0 & 0 & 0 & 0 & 0 & 47442 & \cellcolor{lightgray!35} 341192 \\
    \hline
    \end{tabular}
    \caption{Number of $(m-2)$-metered $(m,n)$-parking functions.}
    \label{tab:m2}
\end{table}

\enumtmtwo

\begin{proof}
    By following Theorem~\ref{thm:chartmtwo}, we need to consider the relation between entries $a_1$, and $a_m$ respectively. 
    We first count the instances where $a_m \leq a_1 \leq k$ by 
    \begin{align}
        \sum_{k=n-m+2}^n \left(\binom{k}{2}+k\right)|Sh_{m-1}(k-1,n-k)|,\label{eq:messy1}
    \end{align} 
    where we start the sum at $n-m+2$, and apply 
    Lemma~\ref{lem:shuff_size} with $m\leq n+1$. 
    Next, we count the instances where $a_1<a_m\leq k$, which can only occur when 
    \begin{align} \label{possiblenewalphas}
        (a_{i_1}, \ldots, a_{i_{m-n+k-2}}) \in Sh_{m-n+k-1}(j-1, k-j-1)
    \end{align} 
    and $a_1 \leq j$, which is precisely the first set of conditions in Theorem~\ref{thm:chartmtwo}. This contributes
    \begin{align}
        &\sum_{k=n-m+2}^n \sum_{j=n-m+1}^{{k-1}}\left(jk -\binom{j}{2} - j\right)\binom{m-2}{n-k}|Sh_{m-n+k-1}(j-1,k-j-1)|(n-k+1)^{n-k-1},\label{eq:help}
    \end{align}
    to the count. 
    Notice that in~\eqref{eq:help}, the factor of $jk-\binom{j}{2}-j$ counts the number of ways to select $a_1$ and $a_m$ such that $a_1\leq j$, $a_m\leq k$, and $a_m>a_1$.
    Also in~\eqref{eq:help}, 
    the factor of $|Sh_{m-n+k-1}(j-1,k-j-1)|$ accounts for the possible options for $(a_{i_1}, \ldots, a_{i_{m-n+k-2}})$ as given in~\eqref{possiblenewalphas}, while the factor of 
    $(n-k+1)^{n-k-1}$ accounts for all the possible $\beta$ of length $n-k$ such that $\beta - (k)^{n-k} \in \PF_{n-k,n-k}$.
    We then need to account for the number of ways to shuffle two words of length $m-n+k-2$ and $n-k$, which can be done in 
    $\binom{m-2}{n-k}$ ways.
    Now we must account for all possible values of $j\in[k-1]$, noting that if $j<n-m+1$, by Lemma~\ref{lem:shuff_size}, the shuffle factor would be zero.

    Our goal is to simplify the sum of~\eqref{eq:messy1} and~\eqref{eq:help}, and expand via Lemma~\ref{lem:shuff_size}.
    To begin, we separate out the $k=n-m+2$ terms from each of the sums~\eqref{eq:messy1} and~\eqref{eq:help}. Note that in~\eqref{eq:help}, when $k=n-m+2$ the only possible accompanying value of $j$ is $j=n-m+1$. 
    Summing the $k=n-m+2$ terms together simplifies to
    \begin{align}
        & \binom{n-m+3}{2}(m-1)^{m-3} + \left((n-m+1)(n-m+2) - \binom{n-m+2}{2} \right)\binom{m-2}{m-2}(m-1)^{m-3}\nonumber\\
        & = \frac{(n-m+3)(n-m+2)}{2} (m-1)^{m-3} + \frac{(n-m+1)(n-m+2)}{2}(m-1)^{m-3}\nonumber\\
        & = (n-m+2)^2(m-1)^{m-3}.\label{eq:term k=n-m+2}
    \end{align}
    We now use the fact that $\binom{k}{2}+k = \binom{k+1}{2}$ and Lemma~\ref{lem:shuff_size} to get that 
    \begin{multline}\label{thissucked}
        \sum_{k=n-m+3}^n \left(\binom{k}{2}+k\right)|Sh_{m-1}(k-1,n-k)|=\\
        \sum_{k=n-m+3}^n\binom{k+1}{2}\binom{m-2}{n-k}(n-m+2)k^{m-n+k-3}(n-k+1)^{n-k-1}.
    \end{multline}
    Next we observe the following:
    \begin{align}
        &\sum_{k=n-m+3}^n \sum_{j=n-m+1}^{{k-1}}\left(jk -\binom{j}{2} - j\right)\binom{m-2}{n-k}|Sh_{m-n+k-1}(j-1,k-j-1)|(n-k+1)^{n-k-1}\nonumber\\
        &=\sum_{k=n-m+3}^n \sum_{j=n-m+1}^{{k-1}}\left(jk -\binom{j+1}{2}\right)\binom{m-2}{n-k}|Sh_{m-n+k-1}(j-1,k-j-1)|(n-k+1)^{n-k-1}\nonumber\\
        &=\sum_{k=n-m+3}^n \binom{m-2}{n-k}(n-k+1)^{n-k-1}\Bigg[ \label{eq:almosthome}\\
        &\qquad\qquad\left((n-m+1)k -\binom{n-m+2}{2}\right)|Sh_{m-n+k-1}(n-m,k-n+m-2)| \nonumber \\
        &\qquad\qquad\qquad\qquad\qquad+\sum_{j=n-m+2}^{k-1}\left(jk -\binom{j+1}{2}\right)|Sh_{m-n+k-1}(j-1,k-j-1)| \Bigg].\nonumber
    \end{align}
    \noindent Next we simplify the following shuffles using Lemma~\ref{lem:shuff_size}.
    For $j=n-m+1$, we have 
    \begin{align}
        |Sh_{m-n+k-1}(n-m,k-n+m-2)|&=(m-n+k-1)^{m-n+k-3}\label{dagger}\\
        \intertext{and for $j>n-m+1$ we have}
        |Sh_{m-n+k-1}(j-1,k-j-1)|&=\binom{m-n+k-2}{k-1-j}(n-m+1)j^{m-n+j-2}(k-j)^{k-j-2}.\label{daggerdagger}
    \end{align}
    Substituting equations~\eqref{dagger} and~\eqref{daggerdagger} into~\eqref{eq:almosthome},  adding the result with equations~\eqref{thissucked} and~\eqref{eq:term k=n-m+2}, and factoring common terms out of the sums yields the final formula:
    \begin{align*}
        \mpf_{m,n}&(m-2)=(n-m+2)^{{2}}(m-1)^{m-3} \\
        &\quad+\sum_{k=n-m+3}^n \binom{m-2}{n-k}(n-k+1)^{n-k-1}\Bigg[ \binom{k+1}{2}(n-m+2)k^{(m-n+k-3)} \\
        &\quad\quad+\left( k(n-m+1) - \binom{n-m+2}{2} \right)(k-n+m-1)^{k-n+m-3} \\
        &\quad\quad\quad+ \sum_{j=n-m+2}^{k-1} \left(jk - \binom{j+1}{2}\right)\binom{m-2-n+k}{k-1-j}(n-m+1)j^{(j+m-2-n)}(k-j)^{k-j-2} \Bigg].\qedhere
    \end{align*}
\end{proof}

Specializing $n = m-1$ in Theorem~\ref{thm:enumtmtwo} yields a nice combinatorial result.

\mtwocor

\begin{proof}
    It follows from Proposition~\ref{prop:unmetered} (or the proof of Theorem~\ref{thm:chartmtwo}) that if $\alpha \in \MPF_{m,m-1}(m-2)$ then $(a_1, \ldots, a_{m-1}) \in PF_{m-1,m-1}$. As a result, the only spot open when car $m$ tries to park will be the one that car $1$ has left. Thus, if $a_1 = k$ there are $k$ possible values for $a_m$, given a certain $(a_2, \ldots, a_{m-1})$.  
    Here we use Diaconis and Hicks \cite[Corollary 1]{DiaHic2017} original formula for the number of classical parking functions of length $m-1$ that begin with a specific $a_1$ value. 
\end{proof}

The sequence $(\mpf_{m,m-1}(m-2))_{m \geq 2}$ corresponds to OEIS entry \cite[\seqnum{A328694}]{OEIS}.

\begin{example}
   Consider the $(3,3)$-parking function $(2,1,3)$. The $2$-metered $(4,3)$-parking functions of the form $(2,1,3,a_4)$ are $(2,1,3,1)$ and $(2,1,3,2).$ More generally, the $2$-metered $(4,3)$-parking functions have the form $(b_1, b_2, b_3, b_4)$, where $(b_1, b_2, b_3) \in \PF_{3}$ and $b_4 \leq b_1$. 
\end{example}

\section{$(n-1)$-Metered Parking Functions}\label{sec:t=n-1}

In this section, we describe and enumerate $(n-1)$-metered $(m,n)$-parking functions, specifically those where $m \geq n+1$, as other cases are equivalent to unmetered $(m,n)$-parking functions. 
When $m\geq n+1$, unlike $t$-metered $(m,n)$-parking functions where $m \leq n$, the outcome will become a periodic sequence repeating the outcome of the first $n$ cars. 
After car $n$ parks, all of the spots will be full until the first car leaves, so car $n+1$ is forced to park in the spot that car $1$ occupied. Similarly, car $n+2$ must park in the spot that car $2$ occupied. It follows that the outcome of the first $n$ cars in $\alpha \in \MPF_{m,n}(n-1)$ is $(p_1, p_2, \ldots, p_n)$ then the outcome of $\alpha$ will, for some $i\in [n]$, be
$$(p_1, p_2, \ldots, p_n,\ \  p_1, p_2, \ldots, p_n,\ \ \ldots\ \ ,p_1, p_2, \ldots p_i ).$$ 

\begin{example}
    If $n = 4$, $t=3$, and $m = 11$, then $\out{3}{4}(2,4,2,1,1,3,2,1,2,3,3) = (2,4,3,1,2,4,3,1,2,4,3)$. Here we have two copies of $(2,4,3,1)$ in the outcome, followed by the first three entries of an additional copy.
\end{example}
\begin{theorem}\label{thm:charn1}
    Fix $k>0$. 
    Consider a sequence $\alpha = (a_1, a_2,\ldots, a_{n+k})\in[n]^{n+k}$ and let $\out{n-1}{n}(a_1, a_2,\ldots, a_n) = (\pi_1,\pi_2, \ldots, \pi_n)$ be the outcome of $(a_1, \ldots, a_n)$ under the $(n-1)$-metered parking scheme. Then, $\alpha$ is an $(n-1)$-metered $(n+k,n)$-parking function if and only if the non-decreasing rearrangement $a_1' \leq a_2' \leq \cdots \leq a_{n}'$ of $a_1$ through $a_{n}$ satisfies $a_i' \leq i$ for all $i \in [n]$ and $a_{n+i}\leq\pi_{i\mod{n}}$ for all $i \in [k]$.
\end{theorem}

\begin{proof}
    ($\Rightarrow$) Let $\alpha = (a_1, a_2, \ldots, a_{n+k})$ be an $(n-1)$-metered 
    $(n+k,n)$-parking function. 
    Note that the prefix $(a_1, a_2, \ldots, a_n)$ must be an $(n-1)$-metered $(n,n)$-parking function.  
    As detailed in Proposition~\ref{prop:unmetered}, $(a_1, \ldots, a_n)$ is in fact a classical parking function.
    Hence, it follows that $(a_1,a_2,\ldots,a_n)$ satisfies the inequality condition on the non-decreasing rearrangement. Namely, if $(a_1',a_2',\ldots,a_n')$ is the non-decreasing rearrangement of $(a_1,a_2,\ldots,a_n)$, then  $a_i'\leq i$ for all $i\in [n]$.
    After car $n$ parks,  car 1 will leave the street. 
    When car $n+1$ arrives, in order to successfully park on the street, it must park in the spot the car 1 vacated, which is spot $\pi_1$. Note this is the case, as that spot is the only available spot on the street. 
    Hence, car $n+1$ must be have a preference for a spot numbered less than or equal to $\pi_1$. 
    Next,  car 2 will leave, so when car $n+2$ attempts to park the only spot it may park in is spot $\pi_2$ (the spot vacated by car 2), and hence its potential preferences must be spots numbered up to and including spot $\pi_2$. 
    This pattern continues until it potentially repeats when car $n+1$ leaves spot $\pi_1$ open for car $2n+1$, at which point the pattern for the possible parking preferences begins again. This establishes that $a_{n+i}\leq \pi_{i\mod n}$ for all $i\in k$, as desired. \smallskip

    \noindent ($\Leftarrow$) Now, consider $\alpha\in[n]^{n+k}$ to be a sequence such that the non-decreasing rearrangement $a_1' \leq a_2' \leq \cdots \leq a_{n}'$ of $a_1$ through $a_{n}$ satisfies $a_i' \leq i$ for $1 \leq i \leq n$ and $a_{n+i} \leq \pi_{i \mod{n}}$ for each $1 \leq i \leq k$. 
    The first $n$ cars can park because they constitute a classical parking function. 
    The second condition assures that each successive car has a preference less than or equal to the spot that will be open when it arrives. 
    Thus $\alpha$ is an $(n-1)$-metered $(n+k,n)$-parking function.
\end{proof}

\begin{table}[!h]
    \centering
    \begin{tabular}{|c||c|c|c|c|c|c|c| }
    \hline 
    \backslashbox{$m$ cars}{$n$ spots} & \textbf{\quad 1 \quad} & \cellcolor{white}\textbf{\quad 2\quad} & \cellcolor{white}\textbf{\quad 3\quad } & \cellcolor{white}\textbf{\quad4 \quad} & \cellcolor{white}\textbf{\quad 5 \quad} & \cellcolor{white}\textbf{\quad 6 \quad} & \textbf{\quad 7 \quad} \\
    \hline 
    \hline
    \textbf{1} & \cellcolor{lightgray!35}1 & 2 & 3 & 4 & 5 & 6 & 7\\ \hline
    \cellcolor{white}\textbf{2}& 1 & \cellcolor{lightgray!35}3 & 8 & 15 & 24 & 35 & 48 \\
    \hline 
    \cellcolor{white}\textbf{3} & 1  & 4 & \cellcolor{lightgray!35}16 & 50 & 108 & 196 & 320 \\
    \hline 
    \cellcolor{white}\textbf{4} &1  & 6 & 27 & \cellcolor{lightgray!35}125 & 432 & 1029 & 2048 \\
    \hline 
    \cellcolor{white}\textbf{5} &1  & 8 & 48 & 257 & \cellcolor{lightgray!35}1296 & 4802  & 12288\\
    \hline 
    \cellcolor{white}\textbf{6}& 1 & 12 & 96 & 540 & 3156 & \cellcolor{lightgray!35}16807 & 65536\\
    \hline 
    \cellcolor{white}\textbf{7}& 1 & 16 & 162 & 1200 & 7734 & 47442 & \cellcolor{lightgray!35} 262144 \\
    \hline
    \end{tabular}
    \caption{The number of $(n-1)$-metered $(m,n)$-parking functions.}
    \label{tab:n1}
\end{table}
The enumeration of $(n-1)$-metered $(m,n)$-parking functions can be given in terms of the number of classical parking functions with a certain outcome. We state and prove this result next, and conclude with an illustrative example that utilizes the Corollary along with some special cases.

\nplusk

\begin{proof}
    Recall the result of Spiro
    \cite[Theorem 3]{Spiro2019SubsetPF} (with an equivalent result in \cite[Proposition 3.1]{countingKnaples}), which shows
    that the number of parking functions with outcome $\pi$ is given by $\prod_{i=1}^n L_i(\pi)$. 
    The result follows from this and the periodicity of the outcome map for the set of $(n-1)$-metered $(n+k,n)$-parking functions, as described in Theorem~\ref{thm:charn1}.
\end{proof}

\begin{example}
    Consider $n=3, m=5,$ and $t=2$. Hence, $k=2$.
    \vspace{4mm}
    \begin{center}
    \begin{tabular}{|l||c|c|c|c|c|c|}\hline & & & & & &
    \\[-7pt] 
        $\pi\in\mathfrak{S}_3$ &$123$ & $132$ & $213$ & $231$ & $312$ & $321$ \\[3pt]
        \hline\hline & & & & & & \\[-5pt]
        $\prod_{i=1}^n L_i(\pi)$ & $1\cdot2\cdot3=6$ &$1\cdot2\cdot1=2$ &$1\cdot1\cdot3=3$ &$1\cdot2\cdot1=2$
        &$1\cdot1\cdot2=2$ &$1\cdot1\cdot1 =1$\\[5pt]
        \hline & & & & & &\\[-5pt]
        $\prod_{j=1}^2\pi_{j\mod{n}}$ &2 &3 &2 &6 &3 &6\\[5pt]
        \hline
    \end{tabular}\\
    \end{center}
    \vspace{4mm}
    Taking the product along columns of row two and three we have that:
    \[\mpf_{5,3}(2)=6\cdot2 + 2\cdot3 + 3\cdot2 + 2\cdot6 + 2\cdot3 + 1\cdot6=48.\]
\end{example}

This is indeed the number of $2$-metered $(5,3)$-parking functions, see Table~\ref{table_2mpf}. The $n=m+1$ diagonal in Table~\ref{tab:n1} corresponds to OEIS entry \cite[\seqnum{A007334}]{OEIS} which counts the number of spanning trees in the graph $K_{n+1}/e$, which results from contracting an edge $e$ in the complete graph $K_{n+1}$ on n vertices for $n \geq 2$. By Proposition~\ref{prop:unmetered}, our sequence also corresponds with the sequence $(|\PF_{n-1,n}|)_{n \geq 2}$ so it is calculated by $2(n+1)^{n-2}$ which also counts the aforementioned number of spanning trees. Additionally, $\mpf_{n+1,n}(n-1)$ is the same special case as $\mpf_{m,m-1}(m-2)$ so by Corollary~\ref{cor:mtwocor}, we have that $\mpf_{n+1,n}(n-1)$ is equal to the sum of the first entries of parking functions of length $n$. 

\section{Future Work}\label{sec:future}

We now detail some possible directions for future research related to the results presented in the previous sections. 
In Section~\ref{sec:prelim}, we remarked that $t$-metered parking functions are not permutation invariant. One can pose the following problem.
\begin{problem}
    Fix a positive integer $t$. What must be true of a $t$-metered parking function $\alpha\in\MPF_{m,n}(t)$ such that  $\sigma(\alpha)$ is also a $t$-metered parking function for all permutations $\sigma$ of $[m]$.
\end{problem}
As we also remarked, $t$-metered parking functions are not nested in a natural way. Hence we ask:
\begin{question}
    When does $\MPF_{m,n}(t) \subseteq \MPF_{m,n}(t')$? Can one give a characterization of the set of $\alpha\in\MPF_{m,n}(t)$ such that $\alpha\in\MPF_{m,n}(t')$ for $t<t'$ (or $t'<t$)?
\end{question}
We also posed Conjecture~\ref{conj:future work}, which remains an open problem.
In addition, given the lack of set containment for $t$-metered parking functions, we ask the following.
\begin{question}
    Is there a different way to define $t$-metered parking functions so that the sets have the nested property one might expect?
\end{question}

In what follows, we let $\lucky(\alpha)=|\{i\in[m]:\text{car $i$ is a lucky car} \}|.$
We are able to give counting formulas for the number of $t$-metered $(m,n)$-parking functions with exactly $1$ and $m$ lucky cars.

\begin{proposition}\label{prop:tlucky}
    Let $m,n \geq 2$ and $t \leq m-1$. Then,
    \begin{align}
        |\{ \alpha \in \MPF_{m,n}(t): \lucky(\alpha) = 1\}|&=
        (t-1)!(n-m+1)  t^{m-t}\text{, and }\label{eq:oneluckycar} \\
     |\{ \alpha \in \MPF_{m,n}(t): \lucky(\alpha) = m\}|&=\frac{n!}{(n-t)!}(n-t)^{m-t}.\label{eq:mluckycars}
    \end{align}
\end{proposition}

\begin{proof} 
    Let $m,n \geq 2$ and $t\leq m-1$. We establish each formula independently, but the overall key to these proofs is knowing that the first car is always lucky, any subsequent car is lucky if it prefers an unoccupied spot, and unlucky if it prefers an occupied spot.
    \begin{enumerate}[leftmargin=.25in]
        \item The first car is always lucky, so every other car must be unlucky. This means that each car after the first must prefer one of the occupied spots when it parks. The second car must prefer the spot where the first car is parked so it only has one preference option. This pattern continues so that when $2 \leq i \leq t$, car $i$ has $i-1$ preference options. When $t+1 \leq i \leq m,$ there will always be $t$ cars parked when car $i$ parks. In this case, car $i$ has $t$ preference options. By construction, each car will park in the lowest spot above the $t$ cars before it. Thus, the first car must park in spot $n-m+1$ or lower, so the remaining $m-1$ cars have enough spaces to park. The total number of preference options for all cars together is $(n-m+1) \cdot 1 \cdot 2 \cdots (t-1) t^{m-t}=(t-1)!(n-m+1)  t^{m-t}$. 
    
        \item For every car to be lucky, each car can prefer any spot except for the spots occupied by the $t$ cars before them.  For $1 \leq i \leq t$, car $i$ has $n-i+1$ preference options. 
        For $t+1 \leq i \leq m$, car $i$ has $n-t$ preference options. Thus, the total number of preference options is $n(n-1) \cdots (n-t+1)(n-t)^{m-t}$, which simplifies to equation~\eqref{eq:mluckycars}. \qedhere
    \end{enumerate}
\end{proof}

Unfortunately, the techniques used to prove Proposition~\ref{prop:tlucky} do not hold in more generality. However, computationally we observe that for $m=3,4,5$, the set of $1$-metered $(m,n)$-parking functions with $k$ lucky cars is a polynomial in $n$ of degree $k$. Based on those observations we pose the following question. 

\begin{question}
    Fix the following positive integer parameters: $t\geq 1$, $m\geq 2$, and $1\leq k\leq m-1$. 
    Let $a_k(n)=|\{\alpha\in\MPF_{m,n}(t): \lucky(\alpha)=k\}|$. 
    Prove or provide a counter example to the claim that the sequence $\big(a_{k}(n)\big)_{n\geq m-1}$ is given by a polynomial in $n$ of degree $k$. 
\end{question}

Our characterization and enumeration of $1$-metered $(m,n)$-parking functions was restricted to $m\leq n$. Hence, we pose the following problem.

\begin{problem}\label{open:m>n}
    Give a recursive or closed formula for the number of $1$-metered $(m,n)$-parking functions in the case where $m>n$.
\end{problem}

Our main results considered $t$-metered parking functions in the special cases where $t=1,n-1,m-2$. 
It remains an open problem to give new formulas for other $t$ values. 
We state this formally below.
\begin{problem}
    Give a characterization and/or enumeration of the $t$-metered $(m,n)$-parking functions for values of $t$ distinct from $1$, $n-1$, and $m-2$.
\end{problem}

There are numerous generalizations of parking functions including generalizations where cars have varying lengths \cite{assortments,Count_Assortments}, where cars back up whenever they find their preferred parking spot occupied \cite{Christensen2019AGO}, and where cars have a set of preferences among the spots on the street \cite{aguilarfraga2023interval,Spiro2019SubsetPF}. It would be of interest to study any of these parking function generalizations under the $t$-metered parking scheme. 
As this section illustrates, there are numerous avenues for further study. 

\appendix

\section{Data Tables}\label{ap:data}
\begin{table}[H]
    \centering
    \begin{tabular}{|c||c|c|c|c|c|c|c|}
    \hline
    \backslashbox{$t$ meter}{$n$ cars} & \textbf{\quad 1\quad} & \textbf{\quad 2 \quad}  &   \textbf{\quad 3 \quad}  & \textbf{ \quad 4 \quad} &  \textbf{\quad 5 \quad} &  \textbf{\quad 6 \quad} &  \textbf{\quad 7 \quad}   \\
    \hline \hline
    \textbf{1}& 1 & 3 &  21 & 209 & 2640 & 40391 & 726103\\
    \hline
    \textbf{2}&  1 & 3 & 16 & 163 & 2142 & 33961 & 626569\\
    \hline
    \textbf{3}&  1 & 3 & 16 & 125 & 1686 & 27629 & 525594\\
    \hline
    \textbf{4}&  1 & 3 & 16 & 125 & 1296 & 21858 & 430062\\
    \hline
    \textbf{5}&  1 & 3 & 16 & 125 & 1296 & 16807 & 341192\\
    \hline
    \textbf{6}&  1 & 3 & 16 & 125 & 1296 & 16807 & 262144\\
    \hline
    \textbf{7}&  1 & 3 & 16 & 125 & 1296 & 16807 & 262144\\
    \hline
    \end{tabular}
    \caption{Number of $t$-metered $(n,n)$-parking functions.}
    \label{tab:tnn}
\end{table}
\begin{table}[H]
    \centering
    \begin{tabular}{|c||c|c|c|c|c|c|c|}
    \hline 
    \backslashbox{$m$ cars}{$n$ spots}  & \textbf{\quad 1\quad} & \textbf{\quad 2 \quad}  &   \textbf{\quad 3 \quad}  & \textbf{ \quad 4 \quad} &  \textbf{\quad 5 \quad} &  \textbf{\quad 6 \quad} &  \textbf{\quad 7 \quad}    \\
    \hline \hline
    \textbf{1}& 1  & 2 & 3 & 4 & 5 & 6 & 7 \\
    \hline 
    \textbf{2}& 0  &  3 & 8 & 15 & 24 & 35 & 48 \\
    \hline 
    \textbf{3} & 0  & 0 & 16 & 50 & 108 & 196 & 320 \\
    \hline 
    \textbf{4}& 0  & 0 & 27 & 163 & 483 & 1095 & 2131 \\
    \hline 
    \textbf{5}& 0  & 0 & 48 & 514 & 2142 & 6098 & 14170 \\
    \hline 
    \textbf{6} & 0 & 0 & 96 & 1665 & 9496 & 33961 & 94228 \\
    \hline
    \textbf{7}& 0 & 0 &162& 5411 & 42196&  189100& 626569  \\
    \hline
    \end{tabular}
    \caption{Number of $2$-metered $(m,n)$-parking functions.}
    \label{table_2mpf}
\end{table}
\begin{table}[H]
    \centering
    \begin{tabular}{|c||c|c|c|c|c|c|c|}
    \hline
    \backslashbox{$m$ cars}{$n$ spots} & \textbf{\quad 1 \quad } & \textbf{\quad 2\quad} & \textbf{\quad 3 \quad } & \textbf{\quad 4\quad} &  \textbf{\quad 5 \quad} & \textbf{\quad 6 \quad} & \textbf{\quad 7\quad} \\
    \hline \hline
    \textbf{1}& 1 & 2 &  3& 4&  5&  6& 7\\
    \hline 
    \textbf{2}&  0 &  3& 8 &  15& 24 & 35&  48\\
    \hline
    \textbf{3}& 0& 0&  16& 50 & 108 & 196 & 320 \\
    \hline
    \textbf{4}&0 & 0& 0 &  125&  432 & 1029 & 2048 \\
    \hline
    \textbf{5}&0 & 0& 0 & 257 & 1686  &5367  & 13076 \\
    \hline
    \textbf{6}&  0& 0& 0&  540&  6253&  27629& 83069 \\
    \hline
    \textbf{7} & 0&  0& 0 &1200  & 23228 & 140599 &  525594 \\
    \hline
    \end{tabular}
    \caption{Number of $3$-metered $(m,n)$-parking functions.}
    \label{tab:3mpf}
\end{table}

\begin{table}[H]
    \centering
    \begin{tabular}{|c||c|c|c|c|c|c|c|}
    \hline
    \backslashbox{$m$ cars}{$n$ spots} & \textbf{\quad 1 \quad } & \textbf{\quad 2\quad} & \textbf{\quad 3 \quad } & \textbf{\quad 4\quad} &  \textbf{\quad 5 \quad} & \textbf{\quad 6 \quad} & \textbf{\quad 7\quad} \\
    \hline \hline
    \textbf{1}& 1 & 2 & 3 & 4& 5 & 6 & 7 \\
    \hline 
    \textbf{2}& 0  & 3 & 8 &  15&  24& 35& 48 \\ 
    \hline
    \textbf{3}& 0 & 0 & 16 & 50 & 108 & 196 & 320 \\
    \hline
    \textbf{4}& 0& 0& 0 & 125 &  432& 1029&2048 \\
    \hline
    \textbf{5}& 0& 0&  0&  0&  1296 &  4802& 12288 \\
    \hline
    \textbf{6}& 0 & 0 & 0 & 0 & 3156 & 21858 & 73276\\
    \hline
    \textbf{7}& 0 & 0 & 0 & 0 & 7734 & 93526 & 430062 \\
    \hline 
    \end{tabular}
    \caption{Number of $4$-metered $(m,n)$-parking functions.}
    \label{tab:4mpf}
\end{table}

\subsection{Connections to Known OEIS Sequences}

The following rows/columns/diagonals of Tables~\ref{table_1mpf}-\ref{tab:4mpf} have appeared in the OEIS:

\begin{itemize}
    \item The first rows of Tables~\ref{table_1mpf},~\ref{table_2mpf},~\ref{tab:3mpf}, and~\ref{tab:4mpf} are OEIS entry  \cite[\seqnum{A000027}]{OEIS}.  
    \item The second rows of Tables~\ref{table_1mpf},~\ref{table_2mpf},~\ref{tab:3mpf}, and~\ref{tab:4mpf} are OEIS entry \cite[\seqnum{A005563}]{OEIS}.
    \item The second column of Table~\ref{table_1mpf} and the second column of Table~\ref{tab:n1} are OEIS entry \cite[\seqnum{A029744}]{OEIS}. 
    \item The third row of Table~\ref{table_1mpf} is OEIS entry \cite[\seqnum{A242135}]{OEIS}.
    \item The  $n=m+1$ diagonal of Table~\ref{table_1mpf} is  OEIS entry \cite[\seqnum{A097691}]{OEIS}.
    \item The $n=m+2$ diagonal of Table~\ref{table_1mpf} is OEIS entry \cite[\seqnum{A342167}]{OEIS}. 
    \item The $n=m+3$ diagonal of Table~\ref{table_1mpf} is OEIS entry \cite[\seqnum{A342168}]{OEIS}.
    \item The  $n=m-1$ diagonals of Table~\ref{tab:m2} and Table~\ref{tab:n1} are OEIS entry \cite[\seqnum{A328694}]{OEIS}.
    \item The $n=m$ main diagonal of Table~\ref{table_1mpf} and the first row of Table~\ref{tab:tnn} are OEIS entry \cite[\seqnum{A097690}]{OEIS}.
    \item The $n=m+1$ diagonal in Table~\ref{tab:n1} is OEIS entry \cite[\seqnum{A007334}]{OEIS} 
\end{itemize}
\printbibliography
\end{document}